\documentclass[twoside,leqno,11pt]{article}

\usepackage{a4wide}

\usepackage{amsmath}

\usepackage{color}
\usepackage{amsmath,amssymb,bbm,enumerate}
\usepackage{pstricks}

\usepackage{epsfig}
\def\accentsfrancais{applemac}
   \RequirePackage[\accentsfrancais]{inputenc}

   \usepackage{hyperref}
\usepackage{dsfont}
\usepackage[all]{xy}

\usepackage{graphics}
\usepackage{pst-3d,float}
\newtheorem{theorem}{Theorem}[section]

\newtheorem{definition}{Definition}[section]
\newtheorem{lemma}{Lemma}[section]
\newtheorem{proposition}[theorem]{Proposition}
\newtheorem{remark}{Remark}[section]
\newenvironment{proof}{\medskip\noindent{\bf Proof.}\;}{\null\hfill $\Box$\par\medskip }
\title{Hyperbolic wavelet transform: an efficient tool for multifractal analysis of anisotropic textures\footnote{This work has been supported by the ANR grant \emph{AMATIS} (ANR2011 BS01 011 02) and the CNRS, Groupe de Recherche \emph{Analyse Multifractale}. }
\author{ \small \sc P. Abry\thanks{Physics Dept., ENS Lyon, CNRS, UMR5672, Lyon, France.}, M. Clausel\thanks{
       University of Grenoble, CNRS, Laboratoire Jean Kuntzmann UMR 5224, Saint Martin d'H\`eres, France.}, S. Jaffard\thanks{ Universit\'e Paris Est,  LAMA, UMR 8050, Cr\'eteil, France.},
          S.G. Roux\footnotemark[2]  and B.Vedel\thanks{LMBA, Universit\'e de Bretagne Sud, European University of Bretagne, Vannes, France.} }
          }

\date{}

\newsavebox{\fmbox}
\newenvironment{fmpage}[1]
 {\begin{lrbox}{\fmbox}\begin{minipage}{#1}}
 {\end{minipage}\end{lrbox}\fbox{\usebox{\fmbox}}}

\def\rmd{\mathrm{d}}
\def\rme{\mathrm{e}}
\def\rmi{\mathrm{i}}

\begin{document}
\maketitle
\begin{abstract}
Global and local regularities of functions are analyzed in anisotropic function spaces, under a common framework, that of hyperbolic wavelet bases.
Local and directional regularity features are characterized by means of global quantities constructed upon the coefficients of hyperbolic wavelet decompositions.
A multifractal analysis is introduced, that jointly accounts for scale invariance and anisotropy.
Its properties are studied in depth.
\end{abstract}
\section{Introduction}

Natural images  often  display various forms of anisotropy.
For a wide range of applications, anisotropy has been quantified through regularity characteristics and features that strongly differ when measured in different directions.
This is, for instance, the case in medical imaging (osteoporosis, muscular tissues, mammographies,...), cf. e.g.~\cite{bonami:estrade:2003,bierme:meerschaert:scheffler:2009}, hydrology~\cite{ponson:2006}, fracture surfaces analysis~\cite{davies:hall:1999},\ldots.
For  such images, a key issue consists first  in describing, within a suitable framework, the anisotropy of the texture, and  then in defining regularity anisotropy parameters that can actually and efficiently be measured via numerical procedures and further involved into e.g., classification schemes.
This requires the design of a mathematical framework that allows to define and estimate these parameters.
Such a  program can be split into several questions, some of them having already been either solved or, at least, patially addressed.

A first issue lies in introducing global and local notions of anisotropic regularity, which emcompass and extend (isotropic) regularity spaces, such as Sobolev or Besov spaces, and the classical notion of pointwise H\"older regularity.
To model anisotropy, the particular setting of an anisotropic self--similar field driven by two parameters (an {\it anisotropy matrix} and a {\it self--similarity index}) has been introduced and studied in~\cite{bierme:meerschaert:scheffler:2009}, where  it is used as a relevant model to describe osteoporosis.
In~\cite{clausel:vedel:2010}, the question of defining in a proper way the concept of anisotropy of an image in relation to its global regularity has been addressed. It has notably been shown that these two parameters can be recovered without {\it a--priori knowledge} of the characteristics of the model, by studying the global smoothness properties of the process.
Furthermore, some of the properties characterizing anisotropy are revealed by the regularity of the sample paths when analyzed with functional spaces well-adapted to anisotropy: Anisotropic Besov spaces.

This preliminary study thus showed the central role that such spaces should play in the mathematical modeling of random anisotropic textures.
The  introduction  of these spaces traces back to the study of some PDEs, cf. e.g.~\cite{triebel:1978}, for the study of semi-elliptic pseudo-differential operators whose symbols  have different degrees of smoothness along different directions, the reader is also referred to~\cite{aimar:gomez:2012}, and references therein, for a recent use of such spaces for optimal regularity results for the heat equation.
Other types of directional function spaces have also been considered, cf. e.g.~\cite{bownik:2003} for the variant supplied by Hardy directional spaces.

A second crucial issue consists in obtaining a simple characterization of these spaces on a ``dictionary''.
Indeed, the challenge  is to  measure the critical exponent of any image in anisotropic Besov spaces for different anisotropies using the same analysis tool.
Wavelet analysis is well--known to be an efficient tool for measuring smoothness in a large range of functional spaces (cf.~\cite{meyer:1990} for details).
Here, however, the main point is that the anisotropy of the analyzing spaces must not be set a priori to a fixed value but instead be allowed to vary.
Specific  bases are thus looked for, which would serve as a common dictionary for anisotropic Besov spaces with different anisotropies.
It is natural that one should use some form of anisotropic wavelets such as curvelets, bandelets, contourlets, shearlets, ridgelets, or wedgelets (see e.g.~\cite{jacques:duval:chaux:peyre:2012} for an thorough review of these representation systems and a comparison of their properties for image processing); natural criteria of choice being, on the mathematical side, that these variations on isotropic wavelets supply bases for the corresponding anisotropic spaces, and, on the applied side, that practically tractable procedure can be devised and implemented to permit the characterization of real-world data according to these function spaces.

Many authors addressed this problem, and proposed different solutions, depending on the precise definition of anisotropic space they started with, as well as on the anisotropic basis they used, cf. e.g.~\cite{devore:konyagin:temlyakov:1998,hochmuth:2002a,hochmuth:2002b,long:triebel:1979,triebel:2004} and also the recent papers~\cite{garrigos:tabacco:2002} by  G. Garrig\'os and A. Tabacco, and~\cite{haroske:tamasi:2005} by  D. Haroske and E. Tam\'asi, which  contain numerous references on the subject.
Note also that, in several cases, a particular type of anisotropy was considered: Parabolic anisotropy (where  a contraction by $\lambda$ in one direction is associated with a contraction by $\lambda^2$ along the orthogonal direction, \cite{lakhonchai:sampo:sumetkijakan:2012,sampo:sumetkijakan:2009,nualtong:sumetkijakan:2005}, and references therein, in particular for applications to directional regularity), the corresponding dictionaries being in that case curvelets or contourlets (corresponding to the Hart--Smith decomposition in the continuous setting, see~\cite{smith:1998}).
This particular choice of anisotropy was motivated from application to PDEs (see e.g., \cite{guo:labate:2008,candes:demanet:2005} where curvelets and ridgelets are used for the study of Fourier integral operators, with applications to the wave equation) but is no longer justified when dealing with images, where no particular form of anisotropy can be postulated a priori.
To the opposite, figuring out the precise form of anisotropy present in data is part of the issue.
This argument also implies that one should not restrict analysis to tools that match one specific type of anisotropy, but rather that to tools embracing all of them simultaneously, in order to be able to detect that that suits data.

From now on, two possible solutions for this problem will be focused on:
 \begin{itemize}
\item One is supplied by {\em anisotropic Triebel bases}, see~\cite{triebel:2004}, that  are constructed from the standard wavelet case through a multiresolution procedure, tailored to a specific anisotropy.
The collection of these bases does not constitute a frame.
However, for a fixed anisotropy, simple characterizations of anisotropic Besov spaces have been supplied within this system.
Such characterizations can thus be used as a building step to construct a multifractal formalism \cite{benbraiek;benslimane:2011b}.
This is further detailed in Section~\ref{s:multimulti}.

Triebel bases provide a powerful tool to deduce results on anisotropic Besov spaces, for {\it a fixed anisotropy}.
In particular, it enables to show that these spaces are isomorphic to the corresponding isotropic Besov spaces.
Further, some results such as embeddings or profiles of Besov characteristics can be obtained, via the transference method proposed by H. Triebel.
However, when it comes to understand the link between different forms of anisotropy - in term of function spaces by example - this tool remains of limited interest.
Indeed, the knowledge of the expansion of a function in one basis gives a priori neither information about its expansion in an other basis  nor about its belonging to all anisotropic Besov spaces.

\item Another possible decomposition system is supplied by {\em hyperbolic wavelets}, introduced in various settings under different denominations (standard, rectangular or hyperbolic wavelet analysis) notably in image coding (see~\cite{Westerink_P_1989_phd_subband_ci}), numerical analysis (see~\cite{Beylkin_G_1991_j-comm-pure-appl-math_fast_wtna1}, \cite{Beylkin_G_1993_p-symp-appl-math_wavelets_fna}) and in~\cite{devore:konyagin:temlyakov:1998},\cite{hochmuth:2002a} for the purpose of approximation theory.
They are simply defined as tensor products of 1D wavelets, yet allowing different dilations factors along different directions, as opposed to the classical discrete wavelet transform that relies on a single isotropic fixation factor.
This key difference enables the study of anisotropy.
Hyperbolic wavelet basis form a non--redundant system by construction, and contain all possible anisotropies.
Hyperbolic wavelet bases have thus been used in statistics for the purpose of adaptive estimation of multidimensional curves.
Notably, it has been proven in two seminal articles~\cite{neumann:vonsachs:1997} and~\cite{neumann:2000} that nonlinear thresholding of noisy hyperbolic wavelet coefficients leads to (near)--optimal minimax rates of convergence over a wide range of anisotropic smoothness classes.
The reader is also referred to the recent work of F. Autin, G. Claeskens, J.M. Freyermuth \cite{autin:claeskens:freyermuth:2012} where this problem is considered from the maxiset point of view. Other interesting applications of hyperbolic analysis can also be founded in~\cite{ayache:leger:pontier:2002},\cite{ayache:2004},\cite{ayache:xiao:2005},\cite{ayache:roueff:xiao:2009a} and~\cite{ayache:roueff:xiao:2009b} where hyperbolic wavelet decompositions of Fractional Brownian Sheets and Linear Fractional Stable Sheets are given and are used to prove many sample paths properties of these random fields (smoothness properties, Hausdorff dimension of the graph).

The key feature of hyperbolic wavelet bases is that they provide a common dictionary for anisotropic Besov spaces.
This result is stated in Theorem~\ref{th:WCBesov} of Section~\ref{s:besov}:
The critical exponent in anisotropic Besov spaces will be related to  some $\ell^p$ norms of the hyperbolic wavelet coefficients.
These mathematical results yield an efficient method for the detection of anisotropy, as detailed in a companion article, where numerical investigations are conducted,  \cite{roux:clausel:vedel:jaffard:abry:2012} .
\end{itemize}

In the present article, it has been chosen to explore the possibilities supplied by the hyperbolic wavelet transform to investigate directional regularity, both in global (anisotropic Besov spaces) and local (directional pointwise regularity) forms.
The underlying motivation is to develop a multifractal formalism relating these two notions (just as the standard multifractal formalism relates the usual Besov spaces with the  notion of (anisotropic) H\"older pointwise smoothness, see~\cite{jaffard:2004}  and references therein).
It also aims at obtaining a numerically stable procedure that thus permits to extract the anisotropic features existing in natural images as well as information related to the size (fractional dimensions) of the  corresponding geometrical sets.

Before proceeding further, let us motivate the choice of  hyperbolic wavelets against Triebel bases.
For a fixed anisotropy, one can argue that Triebel bases display slightly better mathematical advantages: An exact characterization of anisotropic Besov spaces, as shown in~\cite{triebel:2004}, and a characterization of pointwise smoothness as sharp as in the isotropic case, as shown by H. Ben Braiek and M. Ben Slimane in~\cite{benbraiek:benslimane:2011a}.
However, a first purpose of the present contribution is to show that these two important properties hold almost as well for hyperbolic wavelets:
In  Section~\ref{s:besov},  ``almost characterizations'' (i.e., necessary and sufficient conditions that differ by a logarithmic correction) of anisotropic Besov spaces are obtained.
Furthermore, if one is not only interested in analysis, but also in simulation, this slight disadvantage (a logarithmic loss, which in applications can not be detected) is overcompensated by the advantage of using a basis instead of an overcomplete system.
Indeed, generating a random field with prescribed regularity properties requires the use of a basis (using an overcomplete system  cannot guarantee a priori that the  simulated field  with coefficients of specific sizes has the expected properties, since nontrivial linear combinations of the building blocks may vanish).
A contrario, with the hyperbolic wavelet basis, one can easily provide toy examples with different multifractal spectra depending on the anisotropy.
Our being jointly motivated by analysis and synthesis motivates the choice of a system that permits an interesting  trade-off among directional wavelets, in terms of mathematical efficiency and numerical simplicity and robustness, both on the analysis and synthesis sides.
The practical relevance of the mathematical tools introduced and studied here are assessed in a companion paper~\cite{roux:clausel:vedel:jaffard:abry:2012}.  \\

Let us now  further compare Triebel and hyperbolic wavelet bases in terms of pointwise directional smoothness.
First, note that this  notion has been the subject of few investigations so far:
To our knowledge,  the natural definition  which  allows for a wavelet characterization was first introduced by  M. Ben Slimane in the 90s, see~\cite{benslimane:1998}, in order to investigate the multifractal properties of anisotropic selfsimilar functions.
Partial results  when using parabolic basis (i.e., curvelets and Hart--Smith transform) have been obtained by J. Sampo and S. Sumetkijakan see~\cite{lakhonchai:sampo:sumetkijakan:2012,sampo:sumetkijakan:2009,nualtong:sumetkijakan:2005} and references therein.
A generalization and implications in terms of sizes of coefficients on directional wavelets (the so-called ``anisets'', which are a mixture of  of the wavelet and Gabor transform, where the wavelets can be arbitrarily shrunk in certain directions) were  also worked out in~\cite{jaffard:2010}.
Finally, an ``almost ''  characterization of pointwise directional regularity was recently obtained by  H. Ben Braiek and M. Ben Slimane in~\cite{benbraiek;benslimane:2011a} on the Triebel basis coefficients, where the basis is picked so that its  anisotropy parameter is fitted to the type of directional regularity considered.
In Section~\ref{s:besov}, we will obtain a similar result, but relying on the coefficients of the hyperbolic wavelet basis, thus paving the way to the construction of a multifractal formalism.
An important difference with~\cite{benbraiek;benslimane:2011a} is that, here, a single basis fits all anisotropies.
Therefore, as in the case of Besov spaces, the advantage is that no a priori needs to be assumed on the particular considered anisotropy.
This thus can be used as a way to detect the specific anisotropy which exists in data at hand, rather than assuming a priori its particular form beforehand.
Note that other decomposition systems have also been used for the detection of local singularities, see for instance~\cite{donoho:1999,guo:labate:2011} where shearlets and wedgelets are used for the detection of discontinuities along smooth edges.\\

Let us now come back to the anisotropic self--similar  fields considered in~\cite{roux:clausel:vedel:jaffard:abry:2012,clausel:vedel:2010}. Such exactly selfsimilar models are somewhat toy examples, and, though testing regularity indices on their realizations is an important validation step, their study could prove misleadingly simple (just as, in 1D, fractional Brownian motion is too simple a model to fit the richness of situations met in  real-world data).
 Natural images are indeed likely to consist of patchworks of different kinds of deformed pieces and therefore, can be expected to exhibit more complex scale invariance properties, and only in an approximate way.
A natural setting to describe such properties, where different kinds of singularities are mixed up, is supplied by multifractal analysis.
The next step is therefore to combine both anisotropy and multifractal analyses.
To this end, a new form of multifractal analysis is introduced, based on the hyperbolic wavelet coefficients, and relating the global and local characterizations of regularity.
It allows to  take into account both  scale invariance properties and local anisotropic features of an image.
Thus, it  provides a new tool for image classification, seen as a refinement of texture classification based on the usual isotropic multifractal analysis, as proposed for instance in~\cite{abry:jaffard:wendt:2012},\cite{jaffard:2004}.
Section~\ref{s:multimulti} is devoted to the introduction of this new framework: A new multifractal formalism, referred to as the  {\it hyperbolic multifractal formalism}.
It allows to relate local anisotropic regularity of the analyzed image to global quantities called hyperbolic structure functions as commonly done in multifractal analysis.
Note that alternative multifractal analysis and multifractal formalism were introduced by H. Ben Braiek and M. Ben Slimane in~\cite{benbraiek;benslimane:2011b}, based on Triebel basis coefficients.
In their approach, a particular anisotropy is picked, and the corresponding  basis is used.
As above, the main difference between our point of view and theirs is that  we do not pick beforehand a particular anisotropy: Therefore, the approach proposed here does not rely on any a priori assumptions on data, and can thus be used when anisotropies of several types are simultaneously present in data.

Finally, detailed proofs of all the results stated in Sections~\ref{s:besov} and~\ref{s:multimulti}  are provided in Section~\ref{s:proofs}.


\section{Anisotropic global regularity and hyperbolic wavelets}\label{s:besov}

We first focus on the measure of anisotropic global regularity using a common analyzing dictionary: hyperbolic wavelet bases.
Here, we start by providing the reader with a brief account of the corresponding functional spaces.
Thereafter, we recall some well-known facts about hyperbolic wavelet analysis (cf. Section~\ref{s:WCBesov}).
The main result of the present section consists of Theorem~\ref{th:WCBesov}, proven in Section~\ref{s:proofBesov}, which allows to determine the critical directional Besov indices of data by regressions on log-log plot of quantities based on hyperbolic wavelet coefficients (see Section~\ref{s:WCBesov} for a precise statement).


\subsection{Anisotropic Besov spaces}

Anisotropic Besov spaces generalize classical (isotropic) Besov
spaces, and many results concerning isotropic spaces have been extended in this setting, see~\cite{bownik:ho:2005,bownik:2005} for a complete account on
the results used in this section, and~\cite{bownik:2003,triebel:2006} for detailed overviews on anisotropic spaces. Note in particular that these spaces are invariant by smooth  diffeomorphisms on each coordinate, an important requirement for image processing.

Anisotropic Besov spaces verify (asymptotically in the limit of small scales) norm invariances with respect to anisotropic scaling, we, therefore, start by recalling this notion.
Let $\alpha =(\alpha_1,\alpha_2)$ denote a fixed couple of parameters, with $\alpha_1, \, \alpha_2 \ge 0$ and $\alpha_1+\alpha_2 =2$.
In the remainder, such couples will be referred to as {\em admisible anisotropies}.
Such couples quantify the degree of anisotropy of the space ($\alpha_1= \alpha_2=1$ corresponding to the isotropic case).
For any $t\ge 0$ and $\xi=(\xi_1,\xi_2) \in \mathbb{R}^2$, we define anisotropic scaling by $t^{\alpha} \xi = (t^{\alpha_1} \xi_1, t^{\alpha_2}\xi_2)$.  Note that, in this definition and in the following, the coordinate axes are chosen as anisotropy directions. This particular choice can of course be modified by the introduction of an additional rotation (as envisaged e.g., in\cite{roux:clausel:vedel:jaffard:abry:2012}).

Anisotropic Besov spaces may be introduced using  an anisotropic Littlewood Paley analysis, which we now recall.
Let $\varphi_0^{\alpha}  \ge 0$ belong to the Schwartz class ${\mathcal{S}}(\mathbb{R}^2)$ and be such that
$$
\varphi_0^{\alpha}(x) = 1 \quad if \quad \sup_{i=1,2} \vert \xi_{i} \vert \le 1\;,
$$
and
$$
\varphi_0^{\alpha}(x)=0 \quad if \quad \sup_{i=1,2} \vert 2^{-\alpha_{i}} \xi \vert \ge 1\;.
$$
For $j \in \mathbb{N}$, we define
$$
\varphi_j^{\alpha}(x) = \varphi_0^{\alpha}(2^{-j \alpha}\xi)-\varphi_0^{\alpha}(2^{-(j-1)\alpha}\xi)\;.
$$
Then,
$$
\sum_{j=0}^{+\infty} \varphi_j^{\alpha}\equiv 1\;,
$$
and $(\varphi_j^{\alpha})_{j\geq 0}$ is called an {\it anisotropic resolution of the unity}. It satisfies
$$
\mathrm{supp}\left(\varphi_0^{\alpha}\right) \subset R_1^{\alpha}, \quad \mathrm{supp}\left(\varphi_k^{\alpha}\right) \subset R_{j+1}^{\alpha} \setminus R_j^{\alpha}\;,
$$
where
$$
R_j^{\alpha} = \lbrace \xi=(\xi_1,\xi_2) \in \mathbb{R}^2;\, \sup_{i=1,2}\vert \xi_{\ell} \vert \le 2^{\alpha_i k}\rbrace\;.
$$

For $f \in \mathcal{S}'(\mathbb{R}^2)$  let
$$
\Delta^ {\alpha}_j f = {\mathcal{F}}^{-1} \left( \varphi^{\alpha}_j \widehat{f} \right)\;.
$$
The sequence $(\Delta^ {\alpha}_j f )_{j \ge 0})$ is called an {\it anisotropic Littlewood--Paley analysis} of $f$.
The anisotropic Besov spaces are then defined as follows (see~\cite{bownik:ho:2005,bownik:2005}).

\begin{definition}
The Besov space $B^{s,\alpha}_{p,q,|\log|^\beta}(\mathbb{R}^2)$, for $0<p \le +\infty$, $0<q\le + \infty$, $s,\beta \in \mathbb{R}$,  is defined by
$$
B^{s,\alpha}_{p,q,|\log|^\beta}(\mathbb{R}^2) = \lbrace f \in \mathcal{S}'(\mathbb{R}^2); \, \left( \sum_{j \ge 0} j^{-\beta q} 2^{jsq} \Vert \Delta^ {\alpha}_j f \Vert_p^q \right)^{1/q} <+\infty \rbrace\;.
$$
This definition does not depend on the resolution of the chosen unity $\varphi_0^{\alpha}$ and the quantity
$$
\Vert f \Vert_{B^{s,\alpha}_{p,q,|\log|^\beta}} = \left( \sum_{j \ge 0} j^{-\beta q} 2^{jsq} \Vert \Delta^ {\alpha}_jf \Vert_p^q \right)^{1/q}\;,
$$
is a norm (resp., quasi-norm) on $B^{s,\alpha}_{p,q}(\mathbb{R}^2)$ for $1 \leq p, \, q \leq +\infty$ (resp., $0<p, \, q <1$).
\end{definition}

As in the isotropic case, anisotropic Besov spaces encompass a large class of classical anisotropic functional spaces (see~\cite{triebel:2006} for details).
For example, when $p=q=2$ and $(\alpha_1,\alpha_2)\in\mathbb{Q}^2$ is an admissible anisotropy, let us consider $s>0$ such that $s/\alpha_1$ and $s/\alpha_2$ are both integers, then the anisotropic Sobolev space
\[
H^{s,\alpha}(\mathbb{R}^2)=\{f \in L^2(\mathbb{R}^2)\mbox{ such that }\frac{\partial^{s/\alpha_1}f}{\partial x_1}\in L^2(\mathbb{R}^2)\mbox{ and }\frac{\partial^{s/\alpha_2}f}{\partial x_2}\in L^2(\mathbb{R}^2)\}\;,
\]
coincides with the Besov space $B_{2,2}^{s,\alpha}(\mathbb{R}^2)$.

In the special case where $p=q=\infty$, the  spaces $B^{s,\alpha}_{\infty,\infty}(\mathbb{R}^2)$ are called anisotropic H\"{o}lder spaces and are denoted $\mathcal{C}^{s,\alpha}_{|\log|^u}(\mathbb{R}^{2})$. These spaces also admit a finite difference characterization that we now recall (see also~\cite{triebel:2006} for details).

Let $(e_1,e_2)$ denote the canonical basis of $\mathbb{R}^2$.
For a function $f :\mathbb{R}^{2}\rightarrow \mathbb{R}$, $\ell\in\{1,2\}$ and $t\in\mathbb{R}$ one defines
\[
\Delta^{1}_{t,\ell}f(x)=f(x+t e_\ell)-f(x)\;.
\]
The difference of order $M$, $M\geq 2$, of function $f$, along direction $e_\ell$, is then iteratively defined as
\[
\Delta^M_{t,\ell}f(x)=\Delta_{t,\ell}\Delta^{M-1}_{t,\ell}f(x)\;.
\]
One then has:
\begin{proposition}
Let $\alpha=(\alpha_1,\alpha_2)\in (\mathbb{R}^+_*)^2$ such that $\alpha_1+\alpha_2=2$, $s>0$, $u\in\mathbb{R}$ and $f: \mathbb{R}^2\rightarrow \mathbb{R}$. The function $f$ belongs to the anisotropic H\"{o}lder space $\mathcal{C}^{s,\alpha}_{|\log|^u}(\mathbb{R}^{2})$ if
\[
\|f\|_{L^{\infty}(\mathbb{R}^{2})}+\sum_{\ell=1}^2\sup_{t>0}\frac{\|\Delta^{M_\ell}_{t,\ell}f(x)\|_{L^{\infty}(\mathbb{R}^{2})}}{|t|^{s/\alpha_\ell}|\log(|t|)|^u}<+\infty\;,
\]
where for any $\ell\in\{1,2\}$, $M_\ell=[s/\alpha_\ell]+1$.
\end{proposition}

\subsection{Hyperbolic wavelet characterization of anisotropic Besov spaces}\label{s:WCBesov}

We state our first main result which consists of an hyperbolic wavelet caracterization of anisotropic Besov spaces.

We first need to recall the definition of the hyperbolic wavelet bases as tensorial products of two unidimensional wavelet bases (see~\cite{devore:konyagin:temlyakov:1998}) and second state Theorem~\ref{th:WCBesov}, further proven in Section~\ref{s:proofBesov}.
\begin{definition}
Let $\psi$ denote the unidimensional Meyer wavelet and $\varphi$ the associated scaling function. The hyperbolic wavelet basis is defined as the system $\{\psi_{j_1,j_2,k_1,k_2}, \, (j_1, j_2) \in (\mathbb{Z}^+\cup \{-1\})^2, \, (k_1, k_2) \in \mathbb{Z}^2\}$
where
\begin{itemize}
\item if $j_1,j_2\geq 0$,
$$
\psi_{j_1,j_2,k_1,k_2}(x_1,x_2)= \psi(2^{j_1}x_1-k_1)\psi(2^{j_2}x_2-k_2)\;.
$$
\item if $j_1=-1$ and $j_2\geq 0$
$$
\psi_{-1,j_2,k_1,k_2}(x_1,x_2)= \varphi(x_1-k_1)\psi(2^{j_2}x_2-k_2)\;.
$$
\item if $j_1\ge 0$ and $j_2=-1$
$$
\psi_{j_1,-1,k_1,k_2}(x_1,x_2)= \psi(2^{j_1}x_1-k_1)\varphi(x_2-k_2)\;.
$$
\item if $j_1=j_2=-1$
$$
\psi_{-1,-1,k_1,k_2}(x_1,x_2)= \varphi(x_1-k_1)\varphi(x_2-k_2)\;.
$$
\end{itemize}
For any $f\in\mathcal{S}'(\mathbb{R}^2)$, one then defines its hyperbolic wavelet coefficients as follows:
\begin{eqnarray*}
c_{j_1,j_2,k_1,k_2}&=&2^{j_1+j_2}<f,\psi_{j_1,j_2,k_1,k_2}>\mbox{ if }j_1,j_2\geq 0\;,\\
c_{j_1,-1,k_1,k_2}&=&2^{j_1}<f,\psi_{j_1,j_2,k_1,k_2}>\mbox{ if }j_1\geq 0\mbox{ and }j_2=-1\;,\\
c_{-1,j_2,k_1,k_2}&=&2^{j_2}<f,\psi_{j_1,j_2,k_1,k_2}>\mbox{ if }j_1=-1\mbox{ and }j_2\geq 0\;,\\
c_{-1,-1,k_1,k_2}&=&<f,\psi_{j_1,j_2,k_1,k_2}>\mbox{ if }j_1=j_2=-1\;.
\end{eqnarray*}
\end{definition}

\begin{remark}
We chose a $L^1$-normalization for the wavelet coefficients, known to be well-matching scale invariance.
\end{remark}

The main result of this section is an hyperbolic wavelet characterization of the spaces $B^{s,\alpha}_{p,q,|\log|^\beta}(\mathbb{R}^{2})$, up to a logarithmic correction.
In the sequel, some notations are needed. For any $j=(j_1,j_2)\in (\mathbb{N}\cup\{-1\})^2$, let us define:
\[
\|c_{j_1,j_2,\cdot,\cdot}\|_{\ell^p}=\left(\sum_{(k_1,k_2)\in\mathbb{Z}^2}|c_{j_1,j_2,k_1,k_2}|^p\right)^{1/p}\;.
\]
Let $\alpha=(\alpha_1,\alpha_2)$ be an admissible anisotropy, one also defines the following subsets of $\mathbb{N}^2$
\begin{eqnarray*}
\Gamma_j^{(HL)}(\alpha)&=&\{(j_1,j_2)\in\mathbb{N}^2,\,[(j-1) \alpha_1 ]-1\leq j_1\leq [j\alpha_1]+1\mbox{ and }0\leq j_2\leq [(j-1) \alpha_2 ]-1\},\\
\Gamma_j^{(LH)}(\alpha)&=&\{(j_1,j_2)\in\mathbb{N}^2,\,0\leq j_1\leq [(j-1) \alpha_1 ]-1\mbox{ and }[(j-1) \alpha_2 ]-1\leq j_2\leq [j\alpha_2]+1\},\\
\Gamma_j^{(HH)}(\alpha)&=&\{(j_1,j_2)\in\mathbb{N}^2,\,[(j-1) \alpha_1 ]-1\leq j_1\leq [j\alpha_1]+1\mbox{ and }[(j-1) \alpha_2 ]-1\leq j_2\leq [j\alpha_2]+1\}\;,
\end{eqnarray*}
and
\begin{equation}\label{e:gammaj}
\Gamma_j(\alpha)=\Gamma_j^{(HL)}(\alpha)\cup \Gamma_j^{(LH)}(\alpha)\cup \Gamma_j^{(HH)}(\alpha)\;.
\end{equation}
Let us now state our hyperbolic wavelet characterization of anisotropic Besov spaces:
\begin{theorem}\label{th:WCBesov}
Let $\alpha=(\alpha_1,\alpha_2)$ be an admissible anisotropy, $(s,\beta)\in\mathbb{R}^2$ and $(p,q)\in (0,+\infty]^2$. Let $f\in \mathcal{S}'(\mathbb{R}^2)$.
\begin{enumerate}
\item Set $\beta(p,q)=\max(1/p-1,0)+\max(1-1/q,0)$. If
\[
\left(\sum_{j\in\mathbb{N}_0}j^{q\beta(p,q)-\beta q}2^{jsq}\sum_{(j_1,j_2)\in \Gamma_j(\alpha)}2^{-\frac{(j_1+j_2)q}{p}}\;\|c_{j_1,j_2,\cdot,\cdot}\|_{\ell^p}^q\right)^{1/q}<+\infty\;,
\]
then $f \in B^{s,\alpha}_{p,q,|\log|^\beta}(\mathbb{R}^2)$ (with usual modifications when $q=\infty$).
\item Conversely,
\begin{enumerate}
\item If $q<\infty$ and $f \in B^{s,\alpha}_{p,q,|\log|^\beta}(\mathbb{R}^2)$ then
\[
\left(\sum_{j\in\mathbb{N}_0}j^{-\beta q-1}2^{jsq}\sum_{(j_1,j_2)\in \Gamma_j(\alpha)}2^{-\frac{(j_1+j_2)q}{p}}\;\|c_{j_1,j_2,\cdot,\cdot}\|_{\ell^p}^q\right)^{1/q}<+\infty\;.
\]
\item If $f \in B^{s,\alpha}_{p,\infty,|\log|^\beta}(\mathbb{R}^2)$ then
\[
\max_{j\in\mathbb{N}_0}j^{-\beta }2^{js}\max_{(j_1,j_2)\in \Gamma_j(\alpha)}2^{-\frac{(j_1+j_2)}{p}}\;\|c_{j_1,j_2,\cdot,\cdot}\|_{\ell^p}<+\infty\;.
\]
\end{enumerate}
\end{enumerate}
\end{theorem}
\begin{proof}
Theorem~\ref{th:WCBesov} is proven in Section~\ref{s:proofBesov}.
\end{proof}
In particular, for $p=q=\infty$, the following ``almost'' characterization of anisotropic H\"{o}lder spaces by means of hyperbolic wavelets holds:
\begin{proposition}\label{pro:WCGlobal}
Let $\alpha=(\alpha_1,\alpha_2)$ an admissible anisotropy, $(s,\beta)\in\mathbb{R}^2$ and $f\in\mathcal{S}'(\mathbb{R}^2)$.
\begin{enumerate}[(i)]
\item If $f\in\mathcal{C}^{s,\alpha}(\mathbb{R}^{2})$ then there exists some $C>0$ such that for all $j\in \mathbb{N}\cup\{-1\}$ and any $(j_1,j_2)\in \Gamma_j(\alpha)$,
\begin{equation}\label{e:WCGlobal1}
\|c_{j_1,j_2,\cdot,\cdot}\|_{\ell^\infty}\leq C 2^{-js}\;.
\end{equation}
\item Conversely, assume that there exists some $C>0$ such that for all $j\in \mathbb{N}\cup\{-1\}$ and any $(j_1,j_2)\in \Gamma_j(\alpha)$
\begin{equation}\label{e:WCGlobal2}
\|c_{j_1,j_2,\cdot,\cdot}\|_{\ell^\infty}\leq \frac{2^{- js}}{j}\;,
\end{equation}
then $f\in \mathcal{C}^{s,\alpha}(\mathbb{R}^{2})$.
\end{enumerate}
\end{proposition}
In the special case where $p=q=2$, that is if we consider anisotropic Sobolev spaces, there is no logarithmic correction:
\begin{theorem}\label{th:WCBesov2}
Let $\alpha=(\alpha_1,\alpha_2)$ an admissible anisotropy, $s\in\mathbb{R}$. Let $f\in \mathcal{S}'(\mathbb{R}^2)$. The two following assertions are equivalent:
\begin{enumerate}[(i)]
\item $f \in H^{s,\alpha}(\mathbb{R}^2)=B^{s,\alpha}_{2,2}(\mathbb{R}^2)$.
\item
\[
\left(\sum_{j\in\mathbb{N}_0}2^{2js}\sum_{(j_1,j_2)\in\Gamma_j(\alpha)}2^{-(j_1+j_2)}\;\|c_{j_1,j_2,\cdot,\cdot}\|_{\ell^2}^2\right)^{1/2}<+\infty\;.
\]
\end{enumerate}
\end{theorem}
\begin{proof}
Theorem~\ref{th:WCBesov2} is proven in Section~\ref{s:proofBesov}.
\end{proof}

\section{Hyperbolic multifractal analysis}\label{s:multimulti}

We are now interested in the simultaneous analysis of local regularity properties and of anisotropic features of a function.
To that end, we construct a new multifractal analysis, referred to as the hyperbolic multifractal analysis.
Recall that in the isotropic case, the purpose of multifractal analysis is to provide information on the the pointwise singularities of functions.
Multifractal functions are usually such that their local regularity strongly vary from point to point, so that it is not possible to estimate the regularity
index (defined below) of a function at a given point.
Instead, the relevant information consists of the ``sizes'' of the sets of
points where the regularity index takes the same value.
This ``size'' is mathematically formalized as the Hausdorff dimension. The function
that associates the dimension of the set of points sharing the
same regularity index with this index is referred to as the
spectrum of singularities (or multifractal spectrum).
The goal of a { \em  multifractal formalism} is to provide a method that allows to measure the spectrum of
singularities from quantities that can actually be computed on real-world data.
We extend this approach to the anisotropic setting.
Let us first recall that, in the case where the anisotropy of the analyzing space is fixed, this has already been achieved:
See~\cite{benbraiek;benslimane:2011a} for anisotropic pointwise regularity analysis using Triebel bases and~\cite{benbraiek;benslimane:2011b} for the corresponding anisotropic multifractal formalism.
Here, we generalize these two previous works, providing a multifractal analysis which does not rely on the a priori knowledge of the regularity and takes into account all possible anisotropies.
Note that for a fixed anisotropy, both formalims coincide:
Indeed they are derived from wavelet characterizations of the same functional spaces.
Nevertheless, the formalism based on hyperbolic wavelet allows to deal simultaneously with all possible anisotropies, thus providing more useful algorithms for analyzing real-world data.
In addition, the use of hyperbolic wavelet bases offers the possibility to define and synthesize deterministic and stochastic mathematical objects with prescribed anisotropic behavior.

In Section~\ref{s:pointwise}, the different concepts related to pointwise regularity are first recalled.
An hyperbolic wavelet criterion is then devised in Section~\ref{s:WLanis}.
Our main result, Theorem~\ref{th:WL}, is stated in Section~\ref{s:WLanis} and proven in Section~\ref{s:proofs}.
Hyperbolic wavelet analysis is defined in Section~\ref{s:hypmultiform} and the validity of the proposed multifractal formalism is investigated in Theorem~\ref{th:multform}.

\subsection{Anisotropic pointwise regularity and hyperbolic wavelet analysis}\label{s:pointwise}
\subsubsection{Definitions}

Let us start by recalling the usual notion of pointwise regularity (cf.~\cite{jaffard:2004} for a complete review).
\begin{definition}
Let $f$ be in $L^\infty_{loc}(\mathbb{R}^2)$ and $s>0$. The function $f$ belongs to the space $\mathcal{C}^{s}_{|\log|^\beta}(x_0)$ if there exist some $C>0$, $\delta>0$ and $P_{x_0}$ a polynomial with degree less than $s$ such that
\[
\mbox{if }|x-x_0|\leq \delta,\,|f(x)-P_{x_0}(x)|\leq C|x-x_0|^s\cdot |\log(|x-x_0|)|^\beta\;,
\]
where $|\cdot|$ is the usual Euclidean norm of $\mathbb{R}^{2}$.
If $\beta=0$, the space $\mathcal{C}^{s}_{|\log|^0}(x_0)$ is simply denoted $\mathcal{C}^{s}(x_0)$.
\end{definition}

Anisotropic pointwise regularity is further defined as follows.
Let $P$ denote a polynomial of the form:
\[
P(t_1,t_2)=\sum_{(\beta_1,\beta_2)\in \mathbb{N}^2} a_{\beta_1,\beta_2} t_1^{\beta_1}t_2^{\beta_2}\;,
\]
and let $\alpha=(\alpha_1,\alpha_2)$ be an admissible anisotropy.
The $\alpha$--homogeneous degree of the polynomial $P$ is defined as:
\[
d_\alpha(P)=\sup\{\alpha_1 \beta_1+\alpha_2 \beta_2, a_{\beta_1,\beta_2}\neq 0\};.
\]
Finally, for any $t=(t_1,t_2)\in\mathbb{R}^2$, the $\alpha$--homogeneous norm reads:
\[
|t|_\alpha=|t_1|^{1/\alpha_1}+|t_2|^{1/\alpha_2}\;.
\]
We can now define the spaces $\mathcal{C}^{s,\alpha}_{|\log|^\beta}(x_0)$.
\begin{definition}
Let $f\in L^\infty_{loc}(\mathbb{R}^2)$, $\alpha=(\alpha_1,\alpha_2)$ such that $\alpha_1+\alpha_2=2$, $|\cdot|_{\alpha}$, $s>0$ and $\beta\in\mathbb{R}$. The function $f$ belongs to $ \mathcal{C}^{s,\alpha}_{|\log|^\beta}(x_0)$ if there exists some $C>0$, $\delta>0$ and $P_{x_0}$ a polynomial with $\alpha$--homogeneous degree less than $s$ such that
\[
\mbox{if }|x-x_0|_\alpha\leq \delta,\, |f(x)-P_{x_0}(x)|\leq C|x-x_0|_\alpha^s\cdot |\log(|x-x_0|_\alpha)|^\beta\;.
\]
If $\beta=0$, the space $\mathcal{C}^{s,\alpha}_{|\log|^0}(x_0)$ is simply denoted $\mathcal{C}^{s,\alpha}(x_0)$.
\end{definition}
The anisotropic pointwise exponent of a locally bounded function $f$ at $x_0$ can be thus be defined as:
\begin{equation}\label{e:anisoexp}
h_{f,\alpha}(x_0)=\sup\{s,\,f\in\mathcal{C}^{s,\alpha}(x_0)\}\;.
\end{equation}

The reader is referred to~\cite{benbraiek:benslimane:2011a},\cite{jaffard:2010} for more details about the material of this section.


\subsubsection{An hyperbolic wavelet criterion}\label{s:WLanis}

As in the usual anisotropic setting (see~\cite{jaffard:2004}), the anisotropic pointwise H\"older regularity of a function is closely related to  the decay rate of decay of its wavelet leaders.
The usual definition of wavelet leaders needs to be tuned to to the hyperbolic setting:

For any $(j_1,j_2,k_1,k_2)$, let $\lambda(j_1,j_2,k_1,k_2)$ denote the hyperbolic dyadic cube:
\begin{equation}\label{e:dyadiccube}
\lambda=\lambda(j_1,j_2,k_1,k_2)=[\frac{k_1}{2^{j_1}},\frac{k_1+1}{2^{j_1}}[\times [\frac{k_2}{2^{j_2}},\frac{k_2+1}{2^{j_2}}[\;,
\end{equation}
and let $c_\lambda$  stand for $c_{j_1,j_2,k_1,k_2}$. 
The hyperbolic wavelet leaders $d_\lambda$, associated with the hyperbolic cube $\lambda$, can now be defined as:
\[
d_{\lambda}=\sup_{\lambda'\subset \lambda}|c_{\lambda'}|\;.
\]
For any $x_0=(a,b) \in {\mathbb{R}}^2$, let
\[
3\lambda_{j_1,j_2}(x_0)=[\frac{[2^{j_1}a]-1}{2^{j_1}},\frac{[2^{j_1}a]+2}{2^{j_1}}[\times [\frac{[2^{j_2}b]-1}{2^{j_2}},\frac{[2^{j_2}b]+2}{2^{j_2}}[\;,
\]
(where $[\cdot]$ denotes the integer part) and
\[
d_{j_1,j_2}(x_0)=\sup_{\lambda'\subset 3\lambda_{j_1,j_2}(x_0)}|c_{\lambda'}|\;.
\]
The hyperbolic wavelet leaders criterion for pointwise regularity can now be stated as:
\begin{theorem}\label{th:WL}
Let $s>0$ and $\alpha=(\alpha_1,\alpha_2)\in (\mathbb{R}^{*}_+)^2$ such that $\alpha_1+\alpha_2=2$.
\begin{enumerate}
\item Assume that $f\in \mathcal{C}^{s,\alpha}(x_0)$. There exists some $C>0$ such that for any $j_{1},j_2\in \mathbb{N}\cup \{-1\}$ one has
\begin{equation}\label{e:WL}
|d_{j_1,j_2}(x_0)|\leq C2^{-\max(\frac{j_1}{\alpha_1},\frac{j_2}{\alpha_2})s}\;.
\end{equation}
\item Conversely, assume that $f$ is uniformly H\"{o}lder and that~(\ref{e:WL}) holds, then $f\in \mathcal{C}^{s,\alpha}_{|\log|^2}(x_0)$.
\end{enumerate}
\end{theorem}
Proofs are postponed to Section 4.

\subsection{Anisotropic multifractal analysis}
\subsubsection{Two notions of dimension}
In multifractal analysis, two different notions of dimension
are mainly used: the Hausdorff dimension and the packing dimension, whose definitions are now recalled.

The Hausdorff dimension is defined through the Hausdorff measure (see~\cite{falconer:1990} for details).
The best covering of a set $E\subset \mathbb{R}^d$ with sets subordinated to a diameter $\varepsilon$ can be estimated as follows,
\[
\mathcal{H}_\varepsilon^\delta(E)=\inf\{\sum_{i=1}^\infty |E_i|^\delta,\,E\subset \bigcup_{i=1}^\infty E_i, |E_i|\leq \varepsilon\}.
\]
Clearly, $\mathcal{H}_\varepsilon^\delta$ is an outer measure. The Hausdorff  measure is defined  as the (possibly infinite or vanishing) limit $\mathcal{H}_\epsilon^\delta$ as $\varepsilon$ goes to $0$.



The Hausdorff measure is decreasing as $\delta$ goes to infinity. Moreover, $\mathcal{H}^\delta(E)>0$ implies $\mathcal{H}^{\delta'}(E)=\infty$ if $\delta'<\delta$. The following definition is thus meaningful.
\begin{definition}
The Hausdorff dimension $\mathrm{dim}_H(E)$ of a set $E\subset \mathbb{R}^d$ is defined as follows,
\[
\mathrm{dim}_H(E)=\sup\{\delta:\mathcal{H}^\delta(E)=\infty\}\;.
\]
\end{definition}
With this definition, $\mathrm{dim}_H(\emptyset)=-\infty$.

The box dimension (or Minkowski dimension) is simpler to define and to use than the Hausdorff dimension.
\begin{definition}
Let $E\subset \mathbb{R}^d$. If $\varepsilon>0$, let $N_\varepsilon(E)$ be the smallest number of sets of radius $\varepsilon$ required to cover $E$. The upper box dimension is
\[
\overline{\mathrm{dim}}_B(E)=\limsup_{\varepsilon\to 0} \frac{\log N_\varepsilon(E)}{-\log \varepsilon}.
\]
The lower box dimension is
\[
\underline{\mathrm{dim}}_B(E)=\liminf_{\varepsilon\to 0} \frac{\log N_\varepsilon(E)}{-\log \varepsilon}.
\]
If these two quantities are equal, they define the box dimension $\mathrm{dim}_B(E)$ of $E$.
\end{definition}
A significant limitation of box dimension 
 is that a set and its closure have the same dimension.
The packing dimension (introduced by Tricot, see~\cite{tricot:1991}) has better mathematical properties (e.g.,\ the packing dimension of a countable  union of sets is the supremum of their dimensions).
\begin{definition}
The packing dimension $\mathrm{dim}_P(E)$ of a set $E\subset \mathbb{R}^d$ is defined by
\[
\mathrm{dim}_P(E)=\inf\{\sup_i
\{\overline{\mathrm{dim}}_B(E_i)\}:E\subset\bigcup_{i=1}^\infty E_i\}.
\]
\end{definition}
The following inequality holds for any set $E\subset \mathbb{R}^d$,
\[
\mathrm{dim}_H(E)\leq\mathrm{dim}_P(E).
\]

We now define the hyperbolic spectrum of singularities of a locally bounded function using the Hausdorff dimension.
\begin{definition}
Let $f$ be a locally bounded function and $\alpha=(\alpha_1,\alpha_2)\in (\mathbb{R}_{+}^*)^2$ such that $\alpha_1+\alpha_2=2$. The iso--anisotropic--H\"older set are defined as
\[
E_f(H,\alpha)=\{x\in\mathbb{R}^2,h_{f,\alpha}(x)=H\}
\]
where the anisotropic pointwise H\"older $h_{f,\alpha}(x)$ has been defined in~(\ref{e:anisoexp}).

The hyperbolic spectrum of singularities of $f$ is then defined as:
\[
d: (\mathbb{R}^+ \cup \{\infty\})\times (0,2) \to \mathbb{R}^+ \cup \{-\infty\}
\quad (H,a)\mapsto\mathrm{dim}_H(E_f(H,(a,2-a)))\;.
\]
\end{definition}

\subsubsection{The hyperbolic wavelet leader multifractal formalism}
\label{s:hypmultiform}

It is not always possible to compute theoretically the spectrum of
singularities of a given function.
A multifractal formalism thus consists of a practical procedure that yields (a convex hull of) the function $d$, through the construction of structure functions and the use of the Legendre transform.
In the classical case, these formalisms are variants of a seminal derivation, proposed by Parisi and Frisch in the context of the study of hydrodynamic turbulence~\cite{parisi:frisch:1985}.
The hyperbolic wavelet leader multifractal formalism described below aimed at extending the procedure to where both anisotropy and singularities are studied jointly.

Let us define hyperbolic partition functions of a locally bounded function as:
\begin{equation}\label{eq:sf}
S(j,p,\alpha)=2^{-2j}\sum_{(j_1,j_2)\in \Gamma_j(\alpha)} \sum_{(k_1,k_2)\in \mathbb{Z}^2}d_{j_1,j_2,k_1,k_2}^p\;,
\end{equation}
where $\Gamma_j(\alpha)$ has already been defined in Section~\ref{s:WCBesov} with Eq. (\ref{e:gammaj}).

From the definition of an anisotropic scaling function (or scaling exponents)
\begin{equation}\label{eq:omega}
\omega_f(p,\alpha)=\liminf_{j\to\infty} \frac{\log S(j,p,\alpha)}{\log 2^{-j}},
\end{equation}
let us further define the { \em Legendre hyperbolic spectrum}:
\begin{equation}\label{eq:smf:leg}
{ \mathcal L}_f (H, \alpha) = \inf_{p\in\mathbb{R}^*}\{H p-\omega_f(p,(\alpha,2-\alpha))+2\}\;.
\end{equation}

Qualitatively, the Legendre hyperbolic spectrum and the hyperbolic spectrum of singularities $d_f(H,a)$ are expected to coincide, while
the theorem below actually provides an upper bound relationship.
\begin{theorem}\label{th:multform}
Let
$f$ a uniform H\"{o}lder function. The following inequality holds
\begin{equation}\label{eq:majospectr}
\forall (H,a)\in (\mathbb{R}^{*}_+)\times (0,2),\quad d_f(H,a)\leq { \mathcal L}_f (H, a).  
\end{equation}
\end{theorem}

\begin{definition} Let $f$ a uniform H\"{o}lder function, $H>0$ and $a\in (0,2)$. If
Eq. (\ref{eq:majospectr}) simplifies into an equality, i.e.,
\[
\forall (H,a)\in (\mathbb{R}^{*}_+)\times (0,2),\quad d(H,a)= { \mathcal L}_f (H, \alpha) , 
\]
then $f$ satisfies  the hyperbolic multifractal formalism.
\end{definition}

From an applied perspective, Eqs. (\ref{eq:sf}) , (\ref{eq:omega}) and  (\ref{eq:smf:leg}) constitute the core of the practical procedure enabling to compute the Legendre hyperbolic spectrum from the hyperbolic wavelet leaders computed on the data to be analyzed.
Practical implementation show preliminary satisfactory results, notably, for isotropic function, it is clearly observed that the measured $ { \mathcal L}_f (H, \alpha) $ does not depend on $\alpha$.

\section{Proofs}\label{s:proofs}
\subsection{Proof of Theorem~\ref{th:WCBesov}}\label{s:proofBesov}
\subsubsection{Hyperbolic Littlewood-Paley characterization of $B^{s,\alpha}_{p,q}(\mathbb{R}^2)$}
\label{sec:p1}
Let $\theta_0 \in \mathcal{S}(\mathbb{R},\mathbb{R}^+)$ be supported on $[-2,2]$ such that $\theta_0 =1$ on $[-1,1]$. For any $j \in \mathbb{N}$, let us define
$$
\theta_j = \theta_0(2^j \cdot) - \theta_0(2^{j-1} \cdot)\;.
$$
such that $\sum_{j \ge 0} \theta_j(\cdot) =1$ is a 1--D resolution of the unity.

Observe that, for any $j \ge 1$, $\mathrm{supp}\left(\theta_j\right) \subset \{ 2^{j-1} \le \vert \xi \vert \le 2^{j+1} \}$.

\begin{remark}
In the following, the function $\theta_0$ can be chosen with an arbitrary compact support. It does not change the main results even if technical details of proofs and lemmas have to be adapted. It allows to chose the Fourier transform of a Meyer scaling function for $\theta_0$.
\end{remark}

\begin{definition}
\,
\begin{enumerate}
\item For any $j, \, \ell \ge 0$, and any $\xi=(\xi_1,\xi_2) \in \mathbb{R}^2$ set
$$
\phi_{j_1,j_2}(\xi_1, \xi_2)= \theta_{j_1}(\xi_1) \theta_{j_2}(\xi_2)\;.
$$
For any $j_1, \, j_2 \ge 0$, the function $\phi_{j_1,j_2}$ belongs to $\mathcal{S}(\mathbb{R}^2)$ and is compactly supported on $\{2^{\ell_1} \le \vert \xi_1 \vert \le 2^{\ell_1+1} \} \times \{ 2^{\ell_2} \le \vert \xi_2 \vert \le 2^{\ell_2+1}]$. Further $\sum_{j_1 \ge 0} \sum_{j_2 \ge 0} \phi_{j_1,j_2} = 1$ and $(\phi_{j_1,j_2})_{(j_1,j_2)\in\mathbb{N}^2}$ is called an hyperbolic resolution of the unity.
\item For $f \in \mathcal{S}'(\mathbb{R}^2)$ and $j_1, \, j_2 \ge 0$ set
$$
\Delta_{j_1,j_2} f =  {\mathcal{F}}^{-1} \left( \phi_{j_1,j_2} \widehat{f} \right)\;.
$$
The sequence $((\Delta_{j_1,j_2} f)_{j_1,j_2 \ge 0})$ is called an hyperbolic Littlewood--Paley analysis of $f$.
\end{enumerate}
\end{definition}

In the remainder of the section, we are given $\alpha=(\alpha_1,\alpha_2)$ a fixed anisotropy satisfying $\alpha_1+\alpha_2=2$ and $(\varphi^{\alpha}_j)_{j \ge 0}$ an anisotropic resolution of the unity. One then defines the following functions for any $j\geq 0$,
\begin{equation}\label{def:gj}
g_j^\alpha=\sum_{j_1,j_2\in \Gamma_j(\alpha)}\phi_{j_1,j_2}\;,
\end{equation}
where the sets $\Gamma_j(\alpha)$ have been defined in~(\ref{e:gammaj}).

\begin{remark}
Hyperbolic Littlewood--Paley analysis is used in the definition of spaces of mixed smoothness. We refer to~\cite{schmeisser:triebel:2004}  for a study of these spaces and to~\cite{sickel:ullrich:2009} for their link with tensor products of Besov spaces and their hyperbolic wavelet characterizations.
\end{remark}
We now provide the reader with the following hyperbolic Littlewood--Payley characterization of anisotropic Besov spaces:
\begin{theorem}\label{th:LPChar} Let $s\in\mathbb{R}$ and $(p,q)\in (0,+\infty]^2$.
\begin{enumerate}
\item
\begin{enumerate}
\item\label{condi} If $q<\infty$ and
\begin{equation}\label{e:condi}
\left(\sum_{j \ge 0} j^{q\max(1/p-1,0)+\max(q-1,0)}\cdot j^{-\beta q}2^{jsq}   \sum_{(j_1,j_2)\in \Gamma_j(\alpha)} \left\Vert \Delta_{j_1,j_2}(f)   \right\Vert_p^q \right)^{1/q}  <+\infty\;,
\end{equation}
 then $f \in B^{s,\alpha}_{p,q,|\log|^\beta}(\mathbb{R}^2)$.
 \item\label{condibis} If
 \begin{equation}\label{e:conditer}
\max_{j \ge 0} \left(j^{\max(1/p-1,0)+1}\cdot j^{-\beta}2^{js}   \max_{(j_1,j_2)\in \Gamma_j(\alpha)} \left\Vert \Delta_{j_1,j_2}(f)   \right\Vert_p \right)<+\infty\;,
\end{equation}
then $f \in B^{s,\alpha}_{p,\infty,|\log|^\beta}(\mathbb{R}^2)$.
 \end{enumerate}
\item
\begin{enumerate}
\item\label{condii} If $q<\infty$ and $f \in B^{s,\alpha}_{p,q,|\log|^\beta}(\mathbb{R}^2)$ then
\[
\left(\sum_{j \ge 0} j^{-1}\cdot j^{-\beta q}2^{jsq}  \sum_{(j_1,j_2)\in \Gamma_j(\alpha)}  \left\Vert \Delta_{j_1,j_2}(f)   \right\Vert_p^q \right)^{1/q}  <+\infty\;.
\]
\item\label{condiibis} If $f \in B^{s,\alpha}_{p,\infty,|\log|^\beta}(\mathbb{R}^2)$ then
\[
\max_{j \ge 0} \left( j^{-\beta}2^{js}  \sum_{(j_1,j_2)\in \Gamma_j(\alpha)}  \left\Vert \Delta_{j_1,j_2}(f)   \right\Vert_p \right)  <+\infty\;.
\]
\end{enumerate}
\end{enumerate}
\end{theorem}

The proof of Theorem~\ref{th:LPChar} consists of several steps, beginning with
\begin{lemma}\label{lem:1}
\,
\begin{enumerate}
\item\label{intersi}
For any $j \ge 0$ and any $(j_1,j_2)\not\in \Gamma_j(\alpha)$, one has
\begin{equation}\label{e:intersi}
\mathrm{supp}(\varphi^{\alpha}_j)\cap  \mathrm{supp}(\phi_{j_1,j_2})=\emptyset\;.
\end{equation}
\item\label{intersii} For any $j \ge 0$ and any $\ell\not\in \{j-1,j,j+1\}$, one has
\begin{equation}\label{e:intersii}
\mathrm{supp}(g^{\alpha}_j)\cap  \mathrm{supp}(\varphi_{\ell}^{\alpha})=\emptyset\;.
\end{equation}
\end{enumerate}
\end{lemma}
\begin{proof}
 Point~\ref{intersi} of the lemma is first proved, that is if ($\ell_1\ge L_{\max}^{(1)}+1$ and $\ell_2\ge L_{\max}^{(2)}+1$) or ($\ell_1\le L_{\min}^{(1)}-1$ and $\ell_2\le L_{\min}^{(2)}-1$), then $\mathrm{supp}(\varphi^{\alpha}_j)\cap  \mathrm{supp}(\phi_{j_1,j_2})=\emptyset$. \\

Indeed, if $\xi\in \mathrm{supp}(\varphi^{\alpha}_j)$, then $\xi\in R_{j+1}^\alpha\setminus R_j^\alpha$. Hence, if $\ell_1\ge L_{\max}^{(1)}+1$ and $\ell_2\ge L_{\max}^{(2)}+1$, one has for $i=1,2$, $|2^{-\ell_i}\xi_i|\leq 2^{\alpha_i (j+1)-\ell_i}\leq 2^{\alpha_i-1} $, by assumptions on $\ell_1,\ell_2$. Since $\alpha_1$ or $\alpha_2$ necessarily belongs to $(0,1)$, one has $\xi\not\in \mathrm{supp}(\phi_{\ell_1,\ell_2})$. Hence, $\phi_{\ell_1,\ell_2}(\xi)=0$. The same approach leads to $\phi_{\ell_1,\ell_2}(\xi)=0$ if $\ell_1\le L_{\min}^{(1)}-1$ and $\ell_2\le L_{\min}^{(2)}-1$ and Point~(\ref{intersi}) of Lemma~\ref{lem:1} is obtained.

Point~(\ref{intersii}) of Lemma~\ref{lem:1} can be obtained similarly.
\end{proof}
From Lemma~\ref{lem:1} an intermediate hyperbolic Littelwood Paley characterization of anisotropic Besov spaces is obtained.
\begin{proposition}\label{pro:interm}Let $(p,q)\in (0,+\infty]^2$, $s,\beta\in\mathbb{R}$. the two following assertions are equivalent:
\begin{enumerate}
\item\label{condi} $f \in B^{s,\alpha}_{p,q,|\log|^\beta}(\mathbb{R}^2)$.
\item\label{condii}
\begin{equation}\label{e:condii}
\left(\sum_{j \ge 0} j^{-\beta q}2^{jsq} \left\Vert  \sum_{(j_1,j_2)\in \Gamma_j(\alpha)} [ \Delta_{j_1,j_2}(f)]   \right\Vert_p^q \right)^{1/q}<+\infty\;.
\end{equation}
\end{enumerate}
\end{proposition}
\noindent{\it Proof of Proposition~\ref{pro:interm}.}

Let us first show that $f \in B^{s,\alpha}_{p,q,|\log|^\beta}(\mathbb{R}^2)$ implies Inequality~(\ref{e:condii}) of Proposition~\ref{pro:interm}. To this end, Point~\ref{intersi} of Lemma~\ref{lem:1} is used to deduce that for any $j$
\[
\varphi_j^\alpha \widehat{f}=\varphi_j^\alpha\left(\sum_{(j_1,j_2)\in\mathbb{N}^2}\phi_{j_1,j_2}\right)\widehat{f}=\varphi_j^\alpha\left(\sum_{(j_1,j_2)\in\Gamma_j(\alpha)}\phi_{j_1,j_2}\right)\widehat{f}
=\varphi_j^\alpha\left(g_j^\alpha \widehat{f}\right)\;,
\]
where $g_j^\alpha$ is defined by Equation~(\ref{def:gj}). Observe now that replacing the usual dilation with an anisotropic one gives an anisotropic version of Equation~(13) in Section 1.5.2 in~\cite{triebel:1978}. More precisely assume that we are given $p\in (0,+\infty]$, $b>0$ and $M\in \mathcal{S}(\mathbb{R}^{2})$. There exists some $C>0$ not depending on $b$ nor $M$ such that for any $h\in L^p(\mathbb{R}^{2})$ such that $\mathrm{supp}(\widehat{h})\subset \{\xi\in \mathbb{R}^{2},\,\sup_i|\xi_i|\leq b^{\alpha_i}\}$, one has
\begin{equation}\label{e:convmod}
\|\mathcal{F}^{-1}\left(M\mathcal{F}h\right)\|_{L^p(\mathbb{R}^2)}\leq C\|M(b^\alpha\cdot)\|_{H_2^s(\mathbb{R}^2)}\|h\|_{L^p(\mathbb{R}^2)}
\end{equation}
where $H^s_2$ is the usual Bessel potential space and $s>2(1/\min(p,1)-1/2)$.

Set now $b=2^j$, $M=\varphi_j^\alpha$ and $\widehat{h}=g_j^\alpha \widehat{f}$. Since $\varphi_j^\alpha(2^{j\alpha}\cdot)=\varphi_1^\alpha$,  there exists some $C>0$ not depending on $j$ such that for any $p\in (0,+\infty]$ and any $f\in L^p(\mathbb{R}^{2})$
\[
\|\Delta_j^\alpha f\|_{L^p}\leq C\|\sum_{(j_1,j_2)\in\Gamma_j(\alpha)}\Delta_{j_1,j_2}f\|_{L^p}
=C\|(\mathcal{F}^{-1}g_j^\alpha)*f\|_{L^p}\;.
\]
Then
\begin{equation}\label{ineq:1}
\|f\|_{B^{s,\alpha}_{p,q,|\log|^\beta}}=\left(\sum_{j\geq 0}j^{-\beta q}2^{jsq}\|\Delta_j^\alpha f\|_{L^p}^q\right)^{1/q}
\leq C \left(\sum_{j\geq 0}j^{-\beta q} 2^{jsq}\|(\mathcal{F}^{-1}g_j^\alpha)*f\|_{L^p}^q\right)^{1/q}\;,
\end{equation}
which gives Point~\ref{condi} of Proposition~\ref{pro:interm}.

Let us now prove that if Equation~(\ref{e:condii}) of Proposition~\ref{pro:interm} holds then $f\in B^{s,\alpha}_{p,q,|\log|^\beta}(\mathbb{R}^2)$. Point~\ref{intersii} of Lemma~\ref{lem:1} gives for any $j\geq 0$
\[
g_{j}^\alpha \widehat{f}=g_{j}^\alpha \left(\varphi_{j-1}^\alpha+\varphi_{j}^\alpha+\varphi_{j+1}^\alpha\right)\widehat{f}\;.
\]
Hence, Inequality~(\ref{e:convmod}) applied with $b=2^j$, $M=g_j^\alpha$ and $\widehat{h}=(\varphi_{j-1}^\alpha+\varphi_j^\alpha+\varphi_{j+1}^\alpha) \widehat{f}$ gives the existence of some $C>0$ not depending on $j$ nor $f$ such that for any $p\in (0,+\infty]$
\[
\|\left(\mathcal{F}^{-1}g_{j}^\alpha\right)*f\|_{L^p}\leq c\|g_j^\alpha(2^{j\alpha}\cdot)\|_{H_2^s}\|(\mathcal{F}^{-1}\varphi_{j-1}^\alpha+\mathcal{F}^{-1}\varphi_{j}^\alpha+\mathcal{F}^{-1}\varphi_{j+1}^\alpha)*f\|_{L^p}
\]
Since $\|\cdot\|_{L^p}$ is either a norm or a quasi--norm (according to the value of $p$), there exists some $C>0$ such that
\begin{eqnarray*}
\|\left(\mathcal{F}^{-1}g_{j}^\alpha\right)*f\|_{L^p}\leq & C & \|g_j^\alpha(2^{j\alpha}\cdot)\|_{H_2^s} \\
& & \left(\|(\mathcal{F}^{-1}\varphi_{j-1}^\alpha)*f\|_{L^p}+
\|(\mathcal{F}^{-1}\varphi_{j}^\alpha)*f\|_{L^p}+\|(\mathcal{F}^{-1}\varphi_{j+1}^\alpha)*f\|_{L^p}\right)\;.
\end{eqnarray*}
Let us first bound $\|g_j^\alpha(2^{j\alpha}\cdot)\|_{H_2^s}$. To this end, observe that
\begin{eqnarray*}
\mathcal{F}[g_j^\alpha(2^{j\alpha}\cdot)](\xi) & = & 2^{-j(\alpha_1+\alpha_2)}\widehat{g_j}(2^{-j\alpha}\xi)=2^{-2j}\widehat{g_j}(2^{-j\alpha}\xi) \\
& =&\sum_{(j_1,j_2)\in \Gamma_j(\alpha)}2^{j_1+j_2-2j}\widehat{\theta}_{1}(2^{j_1-j\alpha_1}\,\xi_1)\widehat{\theta}_{1}(2^{j_2-\alpha_2 j}\,\xi_2)\;.
\end{eqnarray*}
Hence
\begin{eqnarray*}
\|g_j^\alpha(2^{j\alpha}\cdot)\|_{H_2^s}^2&=&\int_{\mathbb{R}^2}(1+|\xi|^2)^s\left[\sum_{(j_1,j_2)\in \Gamma_j(\alpha)}2^{j_1+j_2-2j}\widehat{\theta}_{1}(2^{j_1-j\alpha_1}\,\xi_1)\widehat{\theta}_{1}(2^{j_2-\alpha_2 j}\,\xi_2)\right]^2\rmd \xi\\
&\leq&\int_{\mathbb{R}^2}(1+|\xi|^2)^s\left[\sum_{(j_1,j_2)\in \Gamma_j(\alpha)}2^{j_1+j_2-2j}|\widehat{\theta}_{1}(2^{j_1-j\alpha_1}\,\xi_1)||\widehat{\theta}_{1}(2^{j_2-\alpha_2 j}\,\xi_2)|\right]^2\rmd \xi\;.
\end{eqnarray*}
Since $\theta_1\in \mathcal{S}(\mathbb{R})$, for any $M>1$ there exists some $C>0$ such that
\[
|\widehat{\theta}_1(\zeta)|\leq \frac{C_M}{(1+|\zeta|)^{2M}}\;.
\]
Finally
\begin{eqnarray*}
\|g_j^\alpha(2^{j\alpha}\cdot)\|_{H_2^s}^2 &\leq & C_M\int_{\mathbb{R}^2}(1+|\xi|^2)^s \\
& & \hspace{1.5cm} \times \left[\sum_{(j_1,j_2)\in \Gamma_j(\alpha)}
\frac{2^{j_1+j_2-2j}}{(1+|2^{j_1-j\alpha_1}\,\xi_1|)^{2M} \cdot (1+|2^{j_2-j\alpha_2}\,\xi_2|)^{2M}}\right]^2\rmd \xi\\
&\leq& C_M\int_{\mathbb{R}^2}(1+|\xi|^2)^s \\
& & \hspace{1.5cm} \times \left[\sum_{(j_1,j_2)\in \Gamma_j(\alpha)}
\frac{1}{(2^{j\alpha_1-j_1}+|\xi_1|)^{2M} \cdot (2^{j\alpha_2-j_2}+|\xi_2|)^{2M}}\right]^2\rmd \xi\;.
\end{eqnarray*}
By the inequality
\[
(a+b)^2\geq a\max(b,1)\;,
\]
valid for any $a>1$, $b>0$ and applied successively with $a=2^{j\alpha_1-j_1}$ and $b=|\xi_1|$, $a=2^{j\alpha_2-j_2}$ and $b=|\xi_2|$, it comes
\begin{eqnarray*}
\|g_j^\alpha(2^{j\alpha}\cdot)\|_{H_2^s}^2&\leq &  C_M  \int_{\mathbb{R}^2}(1+|\xi|^2)^s \\
& & \times \left[\sum_{(j_1,j_2)\in \Gamma_j(\alpha)}2^{(j_1-j\alpha_1)M}2^{(j_2-j\alpha_2)M}
\times\frac{1}{\max(1,|\xi_1|)^{M}\max(1,|\xi_2|)^{M}}\right]^2\rmd \xi\;.
\end{eqnarray*}
With a $M$ sufficiently large it follows that
\[
\sup_j\left(\|g_j^\alpha(2^{j\alpha}\cdot)\|_{H_2^s}\right)<+\infty\;.
\]
Going back to an upper bound of $\|\left(\mathcal{F}^{-1}g_{j}^\alpha\right)*f\|_{L^p}$,  there exists some $C>0$ such that
\[
\|\left(\mathcal{F}^{-1}g_{j}^\alpha\right)*f\|_{L^p}\leq Cj\left(\|(\mathcal{F}^{-1}\varphi_{j-1}^\alpha)*f\|_{L^p}+
\|(\mathcal{F}^{-1}\varphi_{j}^\alpha)*f\|_{L^p}+\|(\mathcal{F}^{-1}\varphi_{j+1}^\alpha)*f\|_{L^p}\right)\;
\]
and
\begin{equation}\label{ineq:2}
\left(\sum_{j\geq 0}j^{-\beta q}2^{jsq}\|(\mathcal{F}^{-1}g_j^\alpha)*f\|_{L^p}^q\right)^{1/q}\leq C\|f\|_{B^{s,\alpha}_{p,q,|\log|^{\beta}}}=\left(\sum_{j\geq 0}j^{-\beta q} 2^{jsq}\|\Delta_j^\alpha f\|_{L^p}^q\right)^{1/q}\;,
\end{equation}
the last shows that if (\ref{e:condii}) holds then $f\in B^{s,\alpha}_{p,q,|\log|^{\beta}}(\mathbb{R}^2)$.\\

\noindent{\it Proof of Theorem~\ref{th:LPChar}.}
Let us first recall that:
\begin{itemize}
\item For any $p\in (0,+\infty]$, $n\in\mathbb{N}$, and $(f_1,\cdots,f_n)\in L^{p}(\mathbb{R}^2)^n$
\begin{equation}\label{ineq:cvxity1}
\|f_1+\cdots+f_n\|_{L^p}\leq n^{\max(1/p-1,0)}\left(\|f_1\|+\cdots+\|f_n\|\right)
\end{equation}
\item For any $q\in (0,+\infty)$, $n\in\mathbb{N}$, and $(a_1,\cdots,a_n)\in (\mathbb{R}_+)^n$
\begin{equation}\label{ineq:cvxity2}
(a_1+\cdots+a_n)^q\leq n^{\max(q-1,0)}\left(a_1^q+\cdots+a_n^q\right)\;.
\end{equation}
\end{itemize}

Let us now prove the first point of the theorem in the case where $q\neq\infty$. For this, assume that~(\ref{e:condi}) holds and let us prove that $f\in B^{s,\alpha}_{p,q,|\log|^{\beta}}(\mathbb{R}^2)$. By Inequalities~(\ref{ineq:cvxity1}), (\ref{ineq:cvxity2}) and the fact that $\mathrm{Card}(\Gamma_j(\alpha))\leq Cj$ there exists $C>0$ such that
\[
\left\|\sum_{(j_1,j_2)\in \Gamma_j(\alpha)}\Delta_{j_1,j_2}f\right\|_{L^p}^q\leq C j^{q\max(1/p-1,0)+\max(q-1,0)}\sum_{(j_1,j_2)\in \Gamma_j(\alpha)}\|\Delta_{j_1,j_2}f\|_{L^p}^q\;.
\]
Hence,
\begin{eqnarray*}
&&\left(\sum_{j\geq 0}j^{-\beta q}2^{jsq}\left\|\sum_{(j_1,j_2)\in \Gamma_j(\alpha)}\Delta_{j_1,j_2}f\right\|_{L^p}^q\right)^{1/q}\\
&\leq& C \left(\sum_{j\geq 0}j^{q\max(1/p-1,0)+\max(q-1,0)}\cdot j^{-\beta q}2^{jsq}\sum_{(j_1,j_2)\in \Gamma_j(\alpha)}\|\Delta_{j_1,j_2}f\|_{L^p}^q\right)^{1/q}\;.
\end{eqnarray*}
It proves that if~(\ref{e:condi}) holds, one has
\[
\left(\sum_{j\geq 0}j^{-\beta q} 2^{jsq}\|\sum_{(j_1,j_2)\in \Gamma_j(\alpha)}\Delta_{j_1,j_2}f\|_{L^p}^q\right)^{1/q}<\infty\;.
\]
Finally, by Point~(1) of Proposition~\ref{pro:interm}, it comes that $f\in B^{s,\alpha}_{p,q,|\log|^{\beta}}(\mathbb{R}^2)$.

\noindent We now deal with the case $q=\infty$. In this case, we have
\begin{eqnarray*}
\max_{j\geq 0}j^{-\beta }2^{js}\left\|\sum_{(j_1,j_2)\in \Gamma_j(\alpha)}\Delta_{j_1,j_2}f\right\|_{L^p}\leq C\max_{j\geq 0}j^{\max(1/p-1,0)}j^{-\beta }2^{js}\sum_{(j_1,j_2)\in \Gamma_j(\alpha)}\|\Delta_{j_1,j_2}f\|_{L^p}\;.
\end{eqnarray*}
Hence if~(\ref{e:condii}) holds, $f\in B^{s,\alpha}_{p,\infty,|\log|^{\beta}}(\mathbb{R}^2)$.

To prove the converse assertion, let us assume $f\in B^{s,\alpha}_{p,q,|\log|^\beta}(\mathbb{R}^2)$. Observe that for any $j\geq 0$ and any $(j_1,j_2)\in\Gamma_j(\alpha)$, one has
\[
\phi_{j_1,j_2}\widehat{f}= \phi_{j_1,j_2}\left(g_{j-1}^\alpha+g_{j}^\alpha+g_{j+1}^\alpha\right)\widehat{f}\;.
\]
Remark that $\phi_{j_1,j_2}(2^{j\alpha}\cdot)$ is bounded in $H^s_2(\mathbb{R}^2)$ independently of $(j_1,j_2)\in \Gamma_j(\alpha)$. Hence, by ~(\ref{e:convmod}), there exists $C>0$ not depending on $j$ nor $f$ such that for any $(j_1,j_2)\in \Gamma_j(\alpha)$
\begin{eqnarray*}
\|(\mathcal{F}^{-1}\phi_{j_1,j_2})*f\|_{L^p}&\leq& C\left(\|(\mathcal{F}^{-1}g_{j-1}^\alpha)*f\|_{L^p}+\|(\mathcal{F}^{-1}g_{j}^\alpha)*f\|_{L^p}
+\|(\mathcal{F}^{-1}g_{j+1}^\alpha)*f\|_{L^p}\right)\;.
\end{eqnarray*}
Again, two cases have to be distinguished according whether $q\neq \infty$ or $q=\infty$.

Let us consider the case $q<\infty$. Observing that $\mathrm{Card}(\Gamma_j(\alpha))\leq Cj$, we deduce that
\[
\sum_{(j_1,j_2)\in\Gamma_j(\alpha)}\|(\mathcal{F}^{-1}\phi_{j_1,j_2})*f\|_{L^p}^q\leq Cj \sum_{l=j-1}^{j+1} \|(\mathcal{F}^{-1}g_{l}^\alpha)*f\|_{L^p}^q\;.
\]
So
\begin{equation}\label{e:ineqa}
\sum_j j^{-1}j^{-\beta q}2^{js q}\sum_{(j_1,j_2)\in\Gamma_j(\alpha)}\|(\mathcal{F}^{-1}\phi_{j_1,j_2})*f\|_{L^p}^q\leq \sum_j j\cdot j^{-1}j^{-\beta q}2^{js q}\|(\mathcal{F}^{-1}g_{j}^\alpha)*f\|_{L^p}^q\;.
\end{equation}
Since in addition the function $f$ is assumed to belong to $B^{s,\alpha}_{p,q,|\log|^\beta}(\mathbb{R}^2)$, one has
\[
\sum_j j^{-\beta q}2^{js q}\|(\mathcal{F}^{-1}g_{j}^\alpha)*f\|_{L^p}^q=\sum_j j\cdot j^{-1}j^{-\beta q}2^{js q}\|(\mathcal{F}^{-1}g_{j}^\alpha)*f\|_{L^p}^q<\infty\;,
\]
which directly yields the required inequality using~(\ref{e:ineqa}).

In the case  $q=\infty$, we have
\begin{eqnarray*}
\max_{(j_1,j_2)\in \Gamma_j(\alpha)}\|(\mathcal{F}^{-1}\phi_{j_1,j_2})*f\|_{L^p}^q&\leq& C\;\max_{\ell=j-1,j,j+1}\|(\mathcal{F}^{-1}g_{\ell}^\alpha)*f\|_{L^p}\;.
\end{eqnarray*}
which leads  for some $C>0$ to
\[
\max_{j\geq 0}\left(j^{-\beta }2^{js}\max_{(j_1,j_2)\in \Gamma_j(\alpha)}\|(\mathcal{F}^{-1}\phi_{j_1,j_2})*f\|_{L^p}\right)\leq C\max_{j\geq 0}\left(j^{-\beta }2^{js}\|(\mathcal{F}^{-1}g_{j}^\alpha)*f\|_{L^p}\right)\;,
\]
that is
\begin{equation}\label{ineq:pointiith}
\max_{j_1,j_2\geq 0}\left(\max(\frac{j_1}{\alpha_1},\frac{j_2}{\alpha_2})\right)^{-\beta}2^{\max(\frac{j_1}{\alpha_1},\frac{j_2}{\alpha_2})s}\|(\mathcal{F}^{-1}\phi_{j_1,j_2})*f\|_{L^p}\leq C\max_{j\geq 0}j^{-\beta }2^{js}\|(\mathcal{F}^{-1}g_{j}^\alpha)*f\|_{L^p}\;.
\end{equation}
Since in addition $f$ is assumed to belong to $B^{s,\alpha}_{p,\infty,|\log|^\beta}(\mathbb{R}^2)$, it comes
\[
\max j^{-\beta }2^{js}\|(\mathcal{F}^{-1}g_{j}^\alpha)*f\|_{L^p}<\infty\;.
\]
Finally, the required conclusion is obtained by an approach similar to the one used for the previous case.

\subsubsection{The special case $p=2$}\label{s:thWC2}
In that case, an {\bf exact} hyperbolic Littlewood--Paley characterization of anisotropic Besov spaces is provided:
\begin{proposition}\label{pro:caracLP2}
Let $f\in\mathcal{S}'(\mathbb{R}^2)$, $s>0$ and $q\in (0,+\infty]$. The two following assertions are equivalent
\begin{enumerate}[(i)]
\item $f \in B^{s,\alpha}_{2,q,|\log|^\beta}(\mathbb{R}^2)$.
\item $\sum\limits_{j \ge 0}j^{-\beta q} 2^{jsq}\left(\sum\limits_{(j_1,j_2)\in \Gamma_j(\alpha)} \Vert \phi_{j_1,j_2} \widehat{f} \Vert^2_{L^2} \right)^{\frac{q}{2}}=\sum\limits_{j \ge 0} j^{-\beta q} 2^{jsq} \left( \sum\limits_{(j_1,j_2)\in \Gamma_j(\alpha)}\Vert \Delta_{j_1,j_2}(f) \Vert^2_{L^2} \right)^{\frac{q}{2}}   <+\infty$.
\end{enumerate}
\end{proposition}
In particular the following exact hyperbolic Littlewood--Paley characterization of anisotropic Sobolev spaces can be stated:
\begin{theorem}\label{pro:caracLP2}
Let $f\in\mathcal{S}'(\mathbb{R}^2)$, $s>0$ and $q\in (0,+\infty]$. The two following assertions are equivalent
\begin{enumerate}[(i)]
\item $f \in H^{s,\alpha}_{|\log|^\beta}(\mathbb{R}^2)=B^{s,\alpha}_{2,2}(\mathbb{R}^2)$.
\item $\sum\limits_{j \ge 0} j^{-2\beta}2^{2js} \sum\limits_{(j_1,j_2)\in \Gamma_j(\alpha)} \Vert \phi_{j_1,j_2} \widehat{f} \Vert^2_{L^2} =\sum\limits_{j \ge 0} j^{-2\beta} 2^{2js} \sum\limits_{(j_1,j_2)\in \Gamma_j(\alpha)}\Vert \Delta_{j_1,j_2}(f) \Vert^2_{L^2}<+\infty$.
\end{enumerate}
\end{theorem}
To prove Proposition~\ref{pro:caracLP2}, let us first precise the relation between anisotropic and hyperbolic resolutions of the unity.
\begin{lemma}\label{lem:1}
For any $j \geq 0$, the following inequality holds on $\mathbb{R}^2$
\begin{equation}\label{e:compLPHyp}
\varphi^{\alpha}_j \le g_j=\sum_{(j_1,j_2)\in\Gamma_j(\alpha)}\phi_{j_1,j_2} \le \varphi^{\alpha}_{j-1}+\varphi^{\alpha}_j + \varphi^{\alpha}_{j+1}\;.
\end{equation}
\end{lemma}
\begin{proof} Let us first observe that
\[
\varphi^{\alpha}_j \le \sum_{(j_1,j_2)\in \mathbb{N}^2} \phi_{j_1,j_2}\;.
\]
To get the left hand side of inequality~(\ref{e:compLPHyp}), we  just have to prove that if $\xi=(\xi_1,\xi_2)\in \mathrm{supp}(\varphi^{\alpha}_j)$, one has $\phi_{j_1,j_2}(\xi)=0$ if $(j_1,j_2)\not\in \Gamma_j$, that is if ($\ell_1\ge L_{\max}^{(1)}+1$ and $\ell_2\ge L_{\max}^{(2)}+1$) or ($\ell_1\le L_{\min}^{(1)}-1$ and $\ell_2\le L_{\min}^{(2)}-1$). \\

Indeed, if $\xi\in \mathrm{supp}(\varphi^{\alpha}_j)$, then $\xi\in R_{j+1}^\alpha\setminus R_j^\alpha$. Hence, if $\ell_1\ge L_{\max}^{(1)}+1$ and $\ell_2\ge L_{\max}^{(2)}+1$, one has for $i=1,2$, $|2^{-\ell_i}\xi_i|\leq 2^{\alpha_i (j+1)-\ell_i}\leq 2^{\alpha_i-1} $, by assumptions on $\ell_1,\ell_2$. Since $\alpha_1$ or $\alpha_2$ necessarily belongs to $(0,1)$, one has $\xi\not\in \mathrm{supp}(\phi_{\ell_1,\ell_2})$. Hence, $\phi_{\ell_1,\ell_2}(\xi)=0$. The same approach leads to $\phi_{\ell_1,\ell_2}(\xi)=0$ if $\ell_1\le L_{\min}^{(1)}-1$ and $\ell_2\le L_{\min}^{(2)}-1$. The left hand side of inequality~(\ref{e:compLPHyp}) is obtained.

Let us now prove that the right hand side of inequality~(\ref{e:compLPHyp}) holds. It comes from the obvious equality
\[
\sum_{(j_1,j_2)\in \Gamma_j} \phi_{j_1,j_2}\leq \sum_{j\ge 0}\varphi^{\alpha}_j\equiv 1 \;.
\]
and if $\xi\in \mathrm{supp}(\phi_{j_1,j_2})$ for some
$(j_1,j_2)\in\Gamma_j(\alpha)$ then for any $\ell\in \{j-1,j,j+1\}$, $\varphi^{\alpha}_{\ell}(\xi)=0$.\\
\end{proof}
Before proving Proposition~\ref{pro:caracLP2}, let us give some background about quasi--orthogonal systems.
\begin{definition}
Let a Hilbert space $H$ and $\langle\cdot,\cdot\rangle$ the associated scalar product. A system $\{f_k, k \in \mathbb{Z}\}$ of $H$ is said to be quasi-orthogonal if there exists some $\ell \in \mathbb{N}$ such that
\begin{equation}\label{e:quasiorth}
\forall (k,k') \in \mathbb{Z}^2, \, \left(\vert k'-k \vert \geq \ell \Rightarrow\quad \langle f_k, f_{k'} \rangle = 0\right)\;.
\end{equation}
\end{definition}
\begin{lemma}\label{lem:qo}
Let $H$ be a Hilbert space and $\|\cdot\|$,$\langle\cdot,\cdot\rangle$  the associated norm and scalar product. Let $\lbrace f_m, \, m \in \mathbb{Z}\}$ a quasi--orthogonal system of $H$ and let $\ell \in \mathbb{N}$ satisfying (\ref{e:quasiorth}). Then
\begin{equation}
\left\Vert \sum_{m \in \mathbb{Z}}  f_m  \right\Vert ^2 \leq (2\ell+1) \sum_{m \in \mathbb{Z}} \Vert f_m \Vert^2\;.
\end{equation}
\end{lemma}
\begin{proof}
Observe that for any $m' \in \mathbb{Z}$, $\langle f_m, f_{m'} \rangle =0$ except if $m'-\ell \le m \le m'+\ell$. Hence
\begin{eqnarray*}
\left\Vert \sum_{m \in \mathbb{Z}}  f_m  \right\Vert ^2 & \le & \sum_{m' \in \mathbb{Z}} \sum_{m \in \mathbb{Z}} \vert \langle f_m, f_{m'} \rangle \vert \\
 & \le & \sum_{m' \in \mathbb{Z}} \sum_{m=m' -\ell}^{m'+\ell} \Vert f_m \Vert \Vert f_{m'} \Vert \\
& \le & \sum_{m' \in \mathbb{Z}} \sqrt{2\ell+1} \left( \sum_{m =m'-\ell}^{m'+\ell} \Vert f_m \Vert^2 \right)^{\frac{1}{2}} \Vert f{_m'} \Vert \\
& \le & \sqrt{2\ell+1} \left( \sum_{m' \in \mathbb{Z}} \left( \sum_{m=m'-\ell}^{m'+\ell} \Vert f_m \Vert^2 \right) \right)^{\frac{1}{2}}. \left( \sum_{m' \in \mathbb{Z}} \Vert f_{m'} \Vert^2 \right)^{\frac{1}{2}}\\
& \le & (2\ell+1)\sum_{m' \in \mathbb{Z}} \Vert f_{m'} \Vert^2\;.
\end{eqnarray*}
\end{proof}
{\bf Proof of Proposition~\ref{pro:caracLP2}}\\
By Plancherel Theorem and by Lemma~\ref{lem:1}, one has
$$
\Vert \Delta_j^{\alpha} f \Vert_{L^2}^2 =  \int_{\mathbb{R}^2}|\varphi^{\alpha}_{j}(\xi)|^2 \vert \widehat{f}(\xi) \vert^2\rmd\xi \le C_0 \int |g_j(\xi)|^2 \vert \widehat{f}(\xi) \vert^2\rmd\xi\;,
$$
and
$$
\int_{\mathbb{R}^2} g_j^2(\xi) \vert \widehat{f}(\xi) \vert^2\rmd\xi  \le   \int_{\mathbb{R}^2} \left[(\varphi^ {\alpha}_{j-1})^2 + (\varphi^{\alpha}_j)^2 + (\varphi^{\alpha}_j)^2 \right] \vert \widehat{f}(\xi) \vert^2 \le
\Vert \Delta_{j-1}^{\alpha} f \Vert_{L^2}^2 + \Vert \Delta_j^{\alpha} f \Vert_{L^2}^2 +\Vert \Delta_{j+1}^{\alpha} f \Vert_{L^2}^2\;,
$$
where $g_j$ is defined by (\ref{e:compLPHyp}). Proposition~\ref{pro:caracLP2} is then a straightforward consequence of the following lemma :
\begin{lemma}
There exists some $C>0$ such that for any $j\ge 0$, one has
\begin{equation}\label{e:equivL12}
C^{-1} \sum_{(j_1,j_2)\in \Gamma_j(\alpha)} \Vert \phi_{j_1,j_2} \widehat{f} \Vert_{L^2}^2
 \le
 \Vert g_j\widehat{f} \Vert_{L^2}^2
 \le   C\sum_{(j_1,j_2)\in \Gamma_j(\alpha)} \Vert \phi_{j_1,j_2} \widehat{f} \Vert_{L^2}^2 \;.
\end{equation}
\end{lemma}
\begin{proof}Since $\mathrm{supp}(\phi_{j_1,j_2})\cap \mathrm{supp}(\phi_{m_1,m_2})=\emptyset$ if $\max(|m_1-j_1|,|m_2-j_2|)\ge 3$, the system $(\phi_{j_1,j_2} \widehat{f})$ is quasi--orthogonal. Hence, by Lemma~\ref{lem:qo}, there exists some $K>0$ such that
$$
\left\Vert \sum_{(j_1,j_2)\in \Gamma_j(\alpha)} \phi_{j_1,j_2}\widehat{f} \right\Vert_{L^2}^2 \le K \sum_{(j_1,j_2)\in \Gamma_j(\alpha)} \Vert \phi_{j_1,j_2}\widehat{f} \Vert_{L^2}^2\;.
$$
For the converse inequality, observe that each term $\langle \phi_{j_1,j_2}\widehat{f},\phi_{j'_1,j'_2}\widehat{f}\rangle=\int \phi_{j_1,j_2}(\xi)\phi_{j'_1,j'_2}(\xi) \vert f(\xi) \vert^2\rmd \xi$ is positive. Hence
\begin{eqnarray*}
\sum_{(j_1,j_2)\in \Gamma_j(\alpha)} \Vert \phi_{j_1,j_2}\widehat{f} \Vert_{L^2}^2 &\le& \sum_{(j_1,j_2)\in \Gamma_j(\alpha)} \langle \phi_{j_1,j_2}\widehat{f},\phi_{j_1,j_2}\widehat{f}\rangle_{L^2} + \sum_{(j_1,j_2)\neq (j'_1,j'_2)\in \Gamma_j(\alpha)} \langle \phi_{j_1,j_2}\widehat{f},\phi_{j'_1,j'_2}\widehat{f}\rangle_{L^2}\\
&=& \left\Vert \sum_{(j_1,j_2)\in \Gamma_j(\alpha)} \phi_{j_1,j_2}\widehat{f} \right\Vert_{L^2}^2\;.
\end{eqnarray*}
\end{proof}
\subsubsection{Proof of the hyperbolic wavelet characterization of anisotropic Besov spaces}
Let us first consider the general case where $(p,q)\in (0,+\infty]^2$,$\beta,s \in\mathbb{R}$ and $\alpha=(\alpha_1,\alpha_2)$ a fixed anisotropy. Intermediate spaces $\mathcal{E}^{s,\alpha}_{p,q,|\log|^\beta}(\mathbb{R}^2)$ are defined as the collection of functions $f$ of $\mathcal{S}'(\mathbb{R}^2)$ such as
$$
\sum_{j \ge 0} j^{-\beta q} 2^{jsq} \sum_{(j_1,j_2) \in \Gamma_j(\alpha)} \Vert \Delta_{j_1,j_2}f \Vert_p^q <+\infty.
$$
A norm on $\mathcal{E}^{s,\alpha}_{p,q,|\log|^\beta}(\mathbb{R}^2)$ is defined as follows
\[
\|f\|_{\mathcal{E}^{s,\alpha}_{p,q,|\log|^\beta}}=\left( \sum_{j \ge 0} j^{-\beta q} 2^{jsq} \sum_{(j_1,j_2) \in \Gamma_j(\alpha)} \Vert \Delta_{j_1,j_2}f \Vert_p^q \right)^{1/q}\;
\]
such that the embeddings
\begin{itemize}
\item if $q<\infty$
\[
\mathcal{E}^{s,\alpha}_{p,q,|\log|^{\beta-\max(1/p-1,0)-\max(1-1/q,0)}}(\mathbb{R}^2)\hookrightarrow B^{s,\alpha}_{p,q,|\log|^\beta}(\mathbb{R}^2)\hookrightarrow\mathcal{E}^{s,\alpha}_{p,q,|\log|^{\beta+1/q}}(\mathbb{R}^2)\;.
\]
\item if $q=\infty$
\[
\mathcal{E}^{s,\alpha}_{p,\infty,|\log|^{\beta-\max(1/p-1,0)-1}}(\mathbb{R}^2)\hookrightarrow B^{s,\alpha}_{p,\infty,|\log|^\beta}(\mathbb{R}^2)\hookrightarrow\mathcal{E}^{s,\alpha}_{p,q,|\log|^{\beta}}(\mathbb{R}^2)\;.
\]
\end{itemize}
are an exact rewriting of Theorem~\ref{th:LPChar}.

In the special case where $p=2$, we proved in Proposition~\ref{pro:caracLP2} that $H^{s,\alpha}_{|\log|^\beta}(\mathbb{R}^2)=B^{s,\alpha}_{2,2,|\log|^\beta}(\mathbb{R}^2)$ and $\mathcal{E}^{s,\alpha}_{2,2,|\log|^{\beta}}(\mathbb{R}^2)$ coincide.

In the following proposition, an hyperbolic wavelet characterization of spaces $\mathcal{E}^{s,\alpha}_{p,q}(\mathbb{R}^2)$ is given. Combining Proposition~\ref{pro:spacesEsp}, Theorems~\ref{th:LPChar} and~\ref{th:caracLP2} directly implies Theorems~\ref{th:WCBesov} and~\ref{th:WCBesov2}.
\begin{proposition}\label{pro:spacesEsp}
Let $(p,q)\in (0,+\infty]^2$, $s,\beta\in\mathbb{R}^2$. The following assertions are equivalent
\begin{enumerate}
\item $f \in \mathcal{E}^{s,\alpha}_{p,q,|\log|^\beta}(\mathbb{R}^2)$
\item $\left( \sum_{j \ge 0} j^{-\beta q} 2^{jsq} \sum_{(j_1,j_2) \in \Gamma_j} 2^{-(j_1+j_2)q/p} \left( \sum_{(k_1,k_2) \in \mathbb{Z}^2} \vert c_{j_1,j_2,k_1,k_2} \vert^p \right)^{\frac{q}{p}}\right)^{\frac{1}{q}} <+\infty$.
\item $\left( \sum_{(j_1,j_2) \in \mathbb{N}_0^2} \left(\max(\frac{j_1}{\alpha_1}, \frac{j_2}{\alpha_2})\right)^{-\beta q} 2^{\left(\max(\frac{j_1}{\alpha_1}, \frac{j_2}{\alpha_2})s-\frac{j_1+j_2}{p}\right)q} \left( \sum_{(k_1,k_2) \in \mathbb{Z}^2} \vert c_{j_1,j_2,k_1,k_2} \vert^p \right)^{\frac{q}{p}} \right)^{\frac{1}{q}} <+\infty$.
\end{enumerate}
\end{proposition}
Let us prove Proposition~\ref{pro:spacesEsp}. The equivalence between assertions $(2)$ and $(3)$  holds since for any $(j_1,j_2)\in \Gamma_j(\alpha)$, one has $\max (\frac{j_1}{\alpha_1}, \frac{j_2}{\alpha_2})+2-2\leq j \leq \max (\frac{j_1}{\alpha_1}, \frac{j_2}{\alpha_2})+2$ and $\cup \Gamma_j = \mathbb{N}_0^2$. The crucial point is the equivalence between assertions $(1)$ and $(2)$ .\\

\noindent{\bf Proof of implication $(1) \Rightarrow (2)$ of Proposition~\ref{pro:spacesEsp}}\\
The proof of this implication relies on the following sampling lemma which is a adaptation of Lemma~2.4 of~\cite{frazier:jawerth:1985} in the case of rectangular support.
\begin{lemma}\label{lem:sample}
Let $p\in (0,+\infty]$ and $j=(j_1,j_2)\in\mathbb{N}_0^2$. Suppose $g\in\mathcal{S}'(\mathbb{R}^2)$ and $\mathrm{supp}(\widehat{g})\subset\{\xi,\,|\xi_1|\leq 2^{j_1+1}\mbox{ and }|\xi_2|\leq 2^{j_2+1}\}$. Then there exists $C>0$ such that
$$
\left( \sum_{(k_1,k_2) \in \mathbb{Z}^2} 2^{-(j_1+j_2)} \left\vert g\left(\frac{k_1}{2^{j_1}}, \frac{k_2}{2^{j_2}}\right) \right\vert^p \right)^{1/p} \le C \Vert g \Vert_{L^p}\;.
$$
\end{lemma}
\begin{proof}
Let $\psi\in\mathcal{S}(\mathbb{R}^2)$ be such that $\mathrm{supp}(\widehat{\psi})\subset \{\xi,\,\max(|\xi_1|,|\xi_2|)\leq \pi\}$ and $\widehat{\psi}\equiv 1$ on $[-2,2]^2$. Set $\psi_j(x)=2^{j_1+j_2}\psi(2^{j_1}x_1,2^{j_2}x_2)$. One  has  $\widehat{\psi_j}\equiv 1$ on $[-2^{j_1+1},2^{j_1+1}]\times [-2^{j_2+1},2^{j_2+1}]$.

By assumption $\mathrm{supp}(\widehat{g})\subset[-2^{j_1+1},2^{j_1+1}]\times [-2^{j_2+1},2^{j_2+1}]$, so that for any $x=(x_1,x_2)\in\mathbb{R}^2$ and any fixed $y=(y_1,y_2)\in\mathbb{R}^2$
\[
g(x+y)=(\psi_j\star g)(x+y)=(2\pi)^{-2}\int_{\xi_1=-2^{j_1+1}}^{2^{j_1+1}}\int_{\xi_2=-2^{j_2+1}}^{2^{j_2+1}}\widehat{\psi_j}(\xi)\widehat{g}(\xi)\rme^{\rmi x\xi}\rme^{\rmi y\xi}\rmd \xi\;.
\]

Denote $\widehat{h_j}$ the periodic extension of $\widehat{\psi_j}$ with period $2^{j_i+1}\pi$ for each variable $\xi_i$ ($i=1,2$). One has
\begin{equation}\label{e:g}
g(x+y)=(2\pi)^{-2}\int_{\xi_1=-2^{j_1+1}}^{2^{j_1+1}}\int_{\xi_2=-2^{j_2+1}}^{2^{j_2+1}}\left(\widehat{h_j}(\xi)\rme^{\rmi x\xi}\right)\left(\widehat{g}(\xi)\rme^{\rmi y\xi}\right)\rmd \xi\;.
\end{equation}
Using an expansion of $\widehat{h_j}\rme^{\rmi x\xi}$ in two dimensional Fourier series, it comes
\begin{eqnarray*}
&& \widehat{h_j}(\xi)\rme^{\rmi x\xi}  \\
& = & \sum_{(\ell_1,\ell_2)\in\mathbb{Z}^2}\left(\int_{\xi_1=-2^{j_1+1}\pi}^{2^{j_1+1}\pi}\int_{\xi_2=-2^{j_2+1}\pi}^{2^{j_2+1}\pi}\widehat{h_j}(\xi)\rme^{\rmi x\xi}\rme^{-\rmi 2^{-j_1}\ell_1\xi_1}\rme^{-\rmi 2^{-j_2}\ell_2\xi_2}\right)\rme^{\rmi 2^{-j_1}\ell_1\xi_1}\rme^{\rmi 2^{-j_2}\ell_2\xi_2}  \\
& =& \sum_{(\ell_1,\ell_2)\in\mathbb{Z}^2}\left(\int_{\xi_1=-2^{j_1+1}\pi}^{2^{j_1+1}\pi}\int_{\xi_2=-2^{j_2+1}\pi}^{2^{j_2+1}\pi}\widehat{\psi_j}(\xi)\rme^{\rmi x\xi}\rme^{-\rmi 2^{-j_1}\ell_1\xi_1}\rme^{-\rmi 2^{-j_2}\ell_2\xi_2}\right)\rme^{\rmi 2^{-j_1}\ell_1\xi_1}\rme^{\rmi 2^{-j_2}\ell_2\xi_2} \\
&=& 2^{-(j_1+j_2)}\sum_{(\ell_1,\ell_2)\in\mathbb{Z}^2}\psi_j(x-2^{-j}\ell)\rme^{\rmi 2^{-j_1}\ell_1\xi_1}\rme^{\rmi 2^{-j_2}\ell_2\xi_2} \;,
\end{eqnarray*}
where for $j=(j_1,j_2)$ and $\ell=(\ell_1,\ell_2)$, the notation $2^{-j}\ell= (2^{-j_1}\ell_1,2^{-j_2}\ell_2)$ is used. Replacing $\widehat{h_j}(\xi)\rme^{\rmi x\xi}$ with the last sum in Equation~(\ref{e:g}) yields that for any $x=(x_1,x_2)\in\mathbb{R}^2$ and any fixed $y=(y_1,y_2)\in\mathbb{R}^2$
\begin{eqnarray*}
& & g(x+y)\\
&=&\frac{2^{-(j_1+j_2)}}{4\pi^2}\sum_{(\ell_1,\ell_2)\in\mathbb{Z}^2}\left(\int_{\xi_1=-2^{j_1+1}}^{2^{j_1+1}}\int_{\xi_2=-2^{j_2+1}}^{2^{j_2+1}}\psi_j(x-2^{-j}\ell)\rme^{\rmi 2^{-j_1}\ell_1\xi_1}\rme^{\rmi 2^{-j_2}\ell_2\xi_2}\left(\widehat{g}(\xi)\rme^{\rmi y\xi}\right)\rmd \xi\right)\\
&=&2^{-(j_1+j_2)}\sum_{(\ell_1,\ell_2)\in\mathbb{Z}^2}g(2^{-j}\ell+y)\psi_j(x-2^{-j}\ell)\;.
\end{eqnarray*}
Hence for all $y\in \lambda_{j_1,j_2,k_1,k_2}=[2^{-j_1}k_1,2^{-j_1}(k_1+1))\times [2^{-j_2}k_2,2^{-j_2}(k_2+1))$
\begin{eqnarray*}
\sup_{|z_1-2^{-j_1}k_1|\leq 2^{-j_1},|z_2-2^{-j_2}k_2|\leq 2^{-j_2}}|g(z)|&\leq& \sup_{|x_1|\leq 2^{-j_1}\sqrt{2},|x_2|\leq 2^{-j_2}\sqrt{2}}|g(x+y)|\\&\leq& 2^{-(j_1+j_2)}\sum_{(\ell_1,\ell_2)\in\mathbb{Z}^2}|g(2^{-j}\ell+y)|\cdot\sup_{\max(2^{j_1}|x_1|,2^{j_2}|x_2|)\leq \sqrt{2}}|\psi_j(x-2^{-j}\ell)|\\
&\leq &2^{-(j_1+j_2)}\sum_{(\ell_1,\ell_2)\in\mathbb{Z}^2}|g(2^{-j}\ell+y)|\cdot\frac{1}{(1+|\ell|)^M}
\end{eqnarray*}
where the last inequality follows from the fast decay of $\psi$. Take $M$ sufficiently large and use either the p triangular inequality either the H\"{o}lder inequality according whether $p\in (0,1)$ or $p\in [1,+\infty]$. Hence, one has
\[
|g(2^{-j_1}k_1,2^{-j_2}k_2)|^p\leq\sup_{|z_1-2^{-j_1}k_1|\leq 2^{-j_1},|z_2-2^{-j_2}k_2|\leq 2^{-j_2}}|g(z)|^p\leq C 2^{-(j_1+j_2)}\sum_{(\ell_1,\ell_2)\in\mathbb{Z}^2}|g(2^{-j}\ell+y)|^p\cdot\frac{1}{(1+|\ell|)^{M'}}\;,
\]
for some $M'>1$.
An integration over $y\in \lambda_{j_1,j_2,k_1,k_2}$ leads to
\[
2^{-(j_1+j_2)}|g(2^{-j_1}k_1,2^{-j_2}k_2)|^p\leq\sum_{(\ell_1,\ell_2)\in\mathbb{Z}^2}\frac{1}{(1+|\ell|)^{M'}}\int_{\lambda_{j_1,j_2,k_1,k_2}}|g(y)|^p \rmd y
\]
and a sum over $k\in \mathbb{Z}^2$ gives
\[
\sum_k 2^{-(j_1+j_2)}|g(2^{-j_1}k_1,2^{-j_2}k_2)|^p\leq \sum_k\sum_{(\ell_1,\ell_2)\in\mathbb{Z}^2}\frac{1}{(1+|\ell|)^{3}}\int_{\lambda_{j_1,j_2,k_1,k_2}}|g(y)|^p \rmd y
\]
which ends the proof of Lemma~\ref{lem:sample}.
\end{proof}

Now, observe that $c_{j_1,j_2,k_1,k_2}=\Delta_{j_1,j_2}f(2^{-j_1}k_1, 2^{-j_2}k_2)$. By Lemma~\ref{lem:sample} applied to $g=\Delta_{j_1,j_2}f\in \mathcal{S}(\mathbb{R}^2)$, one has
$$
\sum_{(k_1,k_2) \in \mathbb{Z}^2}|c_{j_1,j_2,k_1,k_2}|^p=\sum_{(k_1,k_2) \in \mathbb{Z}^2}  \vert \Delta_{j_1,j_2}f(2^{-j_1}k_1, 2^{-j_2}k_2) \vert^ p \le C 2^{j_1} 2^{j_2}\Vert \Delta_{j_1,j_2} f \Vert_p^p\;,
$$
which is the desired wavelet characterization. \\

\noindent{\bf Proof of implication $(2) \Rightarrow (1)$ of Proposition~\ref{pro:spacesEsp}}\\
To obtain the converse implication, the same approach as in the proof of Theorem~3.1 of~\cite{frazier:jawerth:1985} is followed.

Since $\phi_{j_1,j_2}$ and $\psi_{m_1,m_2,k_{1},k_1}$ are both defined as a tensorial product,  Lemma~3.3 of ~\cite{frazier:jawerth:1985} can be applied: there exists some $C>0$ such that for any $\alpha>0$ and for all $x=(x_1,x_2)\in\mathbb{R}^2$ one has
\begin{equation}\label{e:ineqL33}
|\phi_{j_1,j_2}\star \psi_{m_1,m_2,k_{1},k_1}(x)|\leq C \frac{2^{-(|j_1-m_1|+|j_2-m_2|)(M+3)}}{(1+2^{\inf(j_1,m_1)}|x_1-2^{-m_1}k_1|)^{\alpha}(1+2^{\inf(j_2,m_2)}|x_2-2^{-m_2}k_2|)^{\alpha}}\;,
\end{equation}
where $M$ denotes the number of vanishing moments of the wavelets.

A lemma analogous to Lemma~3.4 of~\cite{frazier:jawerth:1985}~ is now proved:
\begin{lemma}\label{lem:34}
Let $p\in [1,+\infty]$, $\ell_1,\ell_2,m_1,m_2$ integers such that $\ell_1\leq m_1$ and $\ell_2\leq m_2$. We are also given some functions $g_{k_1,k_2}$ satisfying the following inequality for some $C>0$~:
\begin{equation}\label{e:ineqL34a}
\forall x=(x_1,x_2)\in\mathbb{R}^2,\,|g_{k_1,k_2}(x)|\leq \frac{C}{(1+2^{\ell_1}|x_1-2^{-m_1}k_1|)^2(1+2^{\ell_2}|x_2-2^{-m_2}k_2|)^2}\;.
\end{equation}
Set
\[
F=\sum_{k=(k_1,k_2)\in\mathbb{Z}^2}d_{k_1,k_2}g_{k_1,k_2}
\]
Then
\begin{equation}\label{e:ineqL34b}
\|F\|_{L^p}\leq C2^{-(m_1+m_2)/p}2^{m_1-\ell_1}2^{m_2-\ell_2}\cdot\left(\sum_{k=(k_1,k_2)\in\mathbb{Z}^2}|d_{k_1,k_2}|^p\right)^{1/p}\;.
\end{equation}
\end{lemma}
\begin{proof}
By definition of the $L^p$--norm, one has~:
\begin{eqnarray*}
\|F\|_{L^p}^p&=&\int_{\mathbb{R}^2}\left|\sum_{k=(k_1,k_2)\in\mathbb{Z}^2}d_{k_1,k_2}g_{k_1,k_2}(x)\right|^p\rmd x\\
&\leq&\sum_{k'=(k'_1,k'_2)\in\mathbb{Z}^2}\int_{\lambda_{m_1,m_2,k'_1,k'_2}}\left|\sum_{k=(k_1,k_2)\in\mathbb{Z}^2}d_{k_1,k_2}g_{k_1,k_2}(x)\right|^p\rmd x\;,
\end{eqnarray*}
where the hyperbolic dyadic cube $\lambda_{m_1,m_2,k'_1,k'_2}$ are defined in~(\ref{e:dyadiccube}). Observe now that, by the usual triangular inequality and by inequality~(\ref{e:ineqL34a}), there exists some $C>0$ such that for any $(k_1,k_2)\in\mathbb{Z}^2$, $(k'_1,k'_2)\in\mathbb{Z}^2$
\[
\sup_{x\in \lambda_{m_1,m_2,k'_1,k'_2}}\left|\sum_{k=(k_1,k_2)\in\mathbb{Z}^2}d_{k_1,k_2}g_{k_1,k_2}(x)\right|\leq \sum_{k=(k_1,k_2)\in\mathbb{Z}^2}\frac{|d_{k_1,k_2}|}{\prod_{i=1,2}(1+2^{\ell_i}|2^{-m_i}k'_i-2^{-m_i}k_i|)^2}
\]
Hence one has
\begin{eqnarray*}
\|F\|_{L^p}^p&\leq&C 2^{-(m_1+m_2)}\sum_{(k'_1,k'_2)\in\mathbb{Z}^2}
\left(\sum_{(k_1,k_2)\in\mathbb{Z}^2}\frac{|d_{k_1,k_2}|}{(1+2^{\ell_1-m_1}|k'_1-k_1|)^2(1+2^{\ell_2-m_2}|k'_2-k_2|)^2}\right)^p\;,
\end{eqnarray*}
Let us recall the usual convolution inequality, valid for any sequences $s,s'$ in $\ell^p(\mathbb{Z}^2)$ for $p\geq 1$,
\[
\|s*s'\|_{\ell_p(\mathbb{Z}^2)}^p\leq \|s\|_{\ell_p(\mathbb{Z}^2)}^p\|s'\|_{\ell^1(\mathbb{Z}^2)}^p\;.
\]
Applied to $s=|d_{k_1,k_2}|$ and $s'=(1+2^{\ell_1-m_1}|k'_1-k_1|)^{-2}(1+2^{\ell_2-m_2}|k'_2-k_2|)^{-2}$, it gives
\begin{eqnarray*}
\|F\|_{L^p}^p&\leq &C 2^{-(m_1+m_2)}\left(\sum_{(k_1,k_2)\in\mathbb{Z}^2}|d_{k_1,k_2}|^p\right)
\left(\sum_{(k'_1,k'_2)\in\mathbb{Z}^2}\frac{1}{(1+2^{\ell_1-m_1}|k'_1|)^2(1+2^{\ell_2-m_2}|k'_2|)^2}\right)^p\;,
\end{eqnarray*}
Recall now the classical result~:
\[
\sum_{k'=(k'_1,k'_2)\in\mathbb{Z}^2}\frac{1}{(1+2^{\ell_1-m_1}|k'_1|)^2(1+2^{\ell_2-m_2}|k'_2|)^2}\leq C2^{m_1-\ell_1}2^{m_2-\ell_2}
\]
Hence
\begin{eqnarray*}
\|F\|_{L^p}^p&\leq& C 2^{-(m_1+m_2)} 2^{(m_1-\ell_1)p}2^{(m_2-\ell_2)p}\left(\sum_{k=(k_1,k_2)\in\mathbb{Z}^2}|d_{k_1,k_2}|^p\right) \\
& &
 \times \left(\sum_{k'=(k'_1,k'_2)\in\mathbb{Z}^2}\frac{1}{\prod_{i=1,2}(1+2^{\ell_i-m_i}|k'_i|)^2}\right)^p\;,
\end{eqnarray*}
which directly yields the required result. It ends the proof of Lemma~\ref{lem:34}.
\end{proof}
Let us now go back to Implication $(2) \Rightarrow (1)$ of Proposition~\ref{pro:spacesEsp}. Two cases are considered: $p\in (0,1)$ and $p\in [1,+\infty]$.\\
Let us first assume that $p\in (0,1)$. \\
We have to bound $\|\Delta_{j_1,j_2}f\|_{L^p}=\|\phi_{j_1,j_2}\star f\|_{L^p}$. Observe that
\[
\phi_{j_1,j_2}\star f=\sum_{m_1,m_2}\sum_{k_1,k_2}c_{m_1,m_2,k_1,k_2}\left(\phi_{j_1,j_2}\star \psi_{m_1,m_2,k_1,k_2}\right)
\]
By the $p$--triangular inequality, it comes
\[
\forall x=(x_1,x_2)\in\mathbb{R}^2,\,|\phi_{j_1,j_2}\star f(x)|^p\leq \sum_{m_1,m_2}\sum_{k_1,k_2}|c_{m_1,m_2,k_1,k_2}|^p\left|(\phi_{j_1,j_2}\star \psi_{m_1,m_2,k_1,k_2})(x)\right|^p
\]
By Inequality~(\ref{e:ineqL33}), for all $x=(x_1,x_2)\in\mathbb{R}^2$, one has
\[
|\phi_{j_1,j_2}\star f(x)|^p\leq \sum_{m_1,m_2}\sum_{k_1,k_2}|c_{m_1,m_2,k_1,k_2}|^p\times\frac{2^{-p(|j_1-m_1|+|j_2-m_2|)(M+3)}}{(1+2^{\inf(j_1,m_1)}|x_1-2^{-m_1}k_1|)^{p\alpha}(1+2^{\inf(j_2,m_2)}|x_2-2^{-m_2}k_2|)^{p\alpha}}
\]
An integration over $\mathbb{R}^2$  implies that~:
\[
\|\phi_{j_1,j_2}\star f(x)\|^p_{L^p}\leq \sum_{m_1,m_2}\sum_{k_1,k_2}|c_{m_1,m_2,k_1,k_2}|^p 2^{-p(|j_1-m_1|+|j_2-m_2|)(M+3)}\;.
\]
Hence
\begin{eqnarray*}
\|f\|_{\mathcal{E}^{s,\alpha}_{p,q,|\log|^\beta}}^q&=&\sum_{j_1,j_2}\left(\max(\frac{j_1}{\alpha_1},\frac{j_2}{\alpha_2})\right)^{-\beta q}2^{q s\max(\frac{j_1}{\alpha_1},\frac{j_2}{\alpha_2})}\|\phi_{j_1,j_2}\star f(x)\|^q_{L^p}\\
&\leq& \sum_{j_1,j_2}\left(\sum_{m_1,m_2}\|c_{m_1,m_2,\cdot,\cdot}\|^p_{\ell^p} 2^{-p(|j_1-m_1|+|j_2-m_2|)(M+3)}\left(\max(\frac{j_1}{\alpha_1},\frac{j_2}{\alpha_2})\right)^{-\beta p}2^{p s\max(\frac{j_1}{\alpha_1},\frac{j_2}{\alpha_2})}\right)^{q/p}
\end{eqnarray*}
Set for any $t\in\mathbb{R}$, $(t)_+=\max(t,0)$ and
\[
\mathrm{sgn}(t)=\left\{\begin{array}{l}1\mbox{ if }t>0,\\0\mbox{ if }t=0,\\-1\mbox{ if }t<0.\end{array}
\right.
\]
Observe now that for any integers $j,m$
\[
m-(m-j)_+\leq j\leq (j-m)_+ + m\;,
\]
and that for any integers $j_1,j_2,m_1,m_2$
\[
\frac{\max(\frac{m_1}{\alpha_1},\frac{m_2}{\alpha_1})}{1-\frac{\max(\frac{(m_1-j_1)_+}{\alpha_1},\frac{(m_2-j_2)_+}{\alpha_1})}{\max(\frac{m_1}{\alpha_1},\frac{m_2}{\alpha_1})}}\leq \max(\frac{j_1}{\alpha_1},\frac{j_2}{\alpha_1})\leq \max(\frac{m_1}{\alpha_1},\frac{m_2}{\alpha_1})\left[1+\max(\frac{(j_1-m_1)_+}{\alpha_1},\frac{(j_2-m_2)_+}{\alpha_1})\right]\;,
\]
(except in the case $m_1=m_2=0$ which can be treated separately). Hence
\[
\|f\|_{\mathcal{E}^{s,\alpha}_{p,q,|\log|^\beta}}^q\leq \sum_{j_1,j_2}\left(\sum_{m_1,m_2} u_{m_1,m_2}v_{j_1-m_1,j_2-m_2}\cdot \right)^{q/p}\;,
\]
with
\[
s_{m_1,m_2}=\left(\max(\frac{m_1}{\alpha_1},\frac{m_2}{\alpha_2})\right)^{-\beta p}2^{p s\max(\frac{m_1}{\alpha_1},\frac{m_2}{\alpha_2})}\|c_{m_1,m_2,\cdot,\cdot}\|^p_{\ell^p}\;,
\]
and
\[
s'_{j_1,j_2}=2^{-p(|j_1|+|j_2|)(M+3)}[1+\max(\frac{(j_1)_+}{\alpha_1},\frac{(j_2)_+}{\alpha_2})]^{-\beta p}2^{\mathrm{sgn}(s) p s\max(\frac{(j_1)_+}{\alpha_1},\frac{(j_2)_+}{\alpha_2})}
\]
If $q/p>1$ Young's inequality can be applied, which states that for any sequences $s,s'$,
\[
\|s*s'\|_{\ell^{q/p}(\mathbb{Z}^2)}\leq \|s\|_{\ell^{q/p}(\mathbb{Z}^2)}\|s'\|_{\ell^1(\mathbb{Z}^2)}\;,
\]
whereas if $q/p\leq 1$  the usual $(q/p)$--triangle inequality and the usual inequality $\|s*s'\|_{\ell^1(\mathbb{Z}^2)}\leq \|s\|_{\ell^1(\mathbb{Z}^2)}\|s'\|_{\ell^1(\mathbb{Z}^2)}$ valid for any sequence $s,s'$ are applied. In any case, the following inequality is obtained
\begin{eqnarray*}
&&\|f\|_{\mathcal{E}^{s,\alpha}_{p,q,|\log|^\beta}}^q\\
&\leq& \left(\sum_{m_1,m_2}\left(\max(\frac{m_1}{\alpha_1},\frac{m_2}{\alpha_2})\right)^{-\beta p}2^{q s\max(\frac{m_1}{\alpha_1},\frac{m_2}{\alpha_2})}\|c_{m_1,m_2,\cdot,\cdot}|^p_{\ell^p}\right)\\ &&\times\sum_{j_1,j_2}\left(2^{-p(j_1+j_2)(M+3)}\left(\max(\frac{j_1}{\alpha_1},\frac{j_2}{\alpha_2})\right)^{-\beta p}2^{p s\max(\frac{j_1}{\alpha_1},\frac{j_2}{\alpha_2})}\right)^{\max(q/p,1)}\;.
\end{eqnarray*}
If the wavelets have sufficiently vanishing moments, we get that
\[
\|f\|_{\mathcal{E}^{s,\alpha}_{p,q,|\log|^\beta}}^q\leq C\left(\sum_{m_1,m_2}\left(\max(\frac{m_1}{\alpha_1},\frac{m_2}{\alpha_2})\right)^{-\beta p}2^{q s\max(\frac{m_1}{\alpha_1},\frac{m_2}{\alpha_2})}\|c_{m_1,m_2,\cdot,\cdot}|^p_{\ell^p}\right)\;,
\]
which is the required result.\\
\noindent We now consider the case $p\in [1,+\infty]$. In this case, observe that
\[
\Delta_{j_1,j_2}f=\sum_{k_1,k_2}d_{k_1,k_2}g_{k_1,k_2}
\]
with
\[
g_{k_1,k_2}=2^{(|j_1-m_1|+|j_2-m_2|)(M+3)}(\phi_{j_1,j_2}\star \psi_{m_1,m_2,k_1,k_2})\;,
\]
and
\[
d_{k_1,k_2}=2^{-(|j_1-m_1|+|j_2-m_2|)(M+3)}c_{j_1,j_2,k_1,k_2}\;.
\]
We set $\ell_1=\inf(j_1,m_1)$ and $\ell_2=\inf(j_2,m_2)$. Lemma~\ref{lem:34} gives
\[
\|\Delta_{j_1,j_2}f\|_{L^p}\leq C 2^{-p(|j_1-m_1|+|j_2-m_2|)(M+3)}2^{-(m_1+m_2)/p}\|c_{m_1,m_2,\cdot,\cdot}|^p_{\ell^p}2^{m_1-\ell_1}2^{m_2-\ell_2}
\]
Again two cases $q\leq 1$ and $q>1$ are distinguished and the same approach than in the case $p\in (0,1)$ is followed. It leads to the required conclusion.

\subsection{Proof of Theorem~\ref{th:WL}}\label{s:proofWC}
First a two--microlocal criterion is proved.
\begin{proposition}\label{pro:WC2micro}
\,
\begin{enumerate}
\item Assume that $f\in \mathcal{C}^{s,\alpha}(x_0)$. Then there exists some $C>0$ such that for any $(j_1,j_2,k_1,k_2)\in (\mathbb{N}\cup \{-1\})^2\times \mathbb{Z}^2$,
\begin{equation}\label{e:WC2micro}
|c_{j_1,j_2,k_1,k_2}|\leq C\min(2^{-\frac{j_1 s}{\alpha_1}}+\left|\frac{k_1}{2^{j_1}}-a\right|^{\frac{s}{\alpha_1}},2^{-\frac{j_2 s}{\alpha_2}}+\left|\frac{k_2}{2^{j_2}}-b\right|^{\frac{s}{\alpha_2}})\;.
\end{equation}
\item Conversely, assume that $f$ is uniformly H\"{o}lder and that~(\ref{e:WC2micro}) holds, then $f\in \mathcal{C}^{s,\alpha}_{|\log|^2}(x_0)$.
\end{enumerate}
\end{proposition}
\begin{proof}
Let us first assume that $f\in \mathcal{C}^{s,\alpha}(x_0)$ with $x_0=(a,b)$. Assume that $j_1\neq -1$ and $j_2\neq -1$. By definition of the hyperbolic wavelet coefficients one has
\begin{eqnarray*}
c_{j_1,j_2,k_1,k_2}&=&2^{j_1+j_2}\int_{\mathbb{R}^2} f(x_1,x_2)\psi(2^{j_1}x_1-k_1)\psi(2^{j_2}x_2-k_2)\rmd x_1\rmd x_2
\end{eqnarray*}
Since $\psi$ admits at least one vanishing moment, the two following equalities hold
\begin{equation}\label{e:eq1}
c_{j_1,j_2,k_1,k_2}=2^{j_1+j_2}\int_{\mathbb{R}^2} (f(x_1,x_2)-P_{x_0}(a,x_2))\psi(2^{j_1}x_1-k_1)\psi(2^{j_2}x_2-k_2)\rmd x_1\rmd x_2
\end{equation}
and
\begin{equation}\label{e:eq2}
c_{j_1,j_2,k_1,k_2}=2^{j_1+j_2}\int_{\mathbb{R}^2} (f(x_1,x_2)-P_{x_0}(x_1,b))\psi(2^{j_1}x_1-k_1)\psi(2^{j_2}x_2-k_2)\rmd x_1\rmd x_2
\end{equation}
Equality~(\ref{e:eq1}) and the assumption $f\in\mathcal{C}^{s,\alpha}(x_0)$ imply that
\begin{eqnarray*}
|c_{j_1,j_2,k_1,k_2}|&\leq& 2^{j_1+j_2}\int |x_1-a|^s_\alpha|\psi(2^{j_1}x_1-k_1)\psi(2^{j_2}x_2-k_2)|\rmd x_1\rmd x_2\\
&\leq&2^{j_1+j_2}\int_{\mathbb{R}^2} \left(\left|x_1-\frac{k_1}{2^{j_1}}\right|^{s/\alpha_1}+\left|\frac{k_1}{2^{j_1}}-a|^{s/\alpha_1}\right)\right|\psi(2^{j_1}x_1-k_1)\psi(2^{j_2}x_2-k_2)|\rmd x_1\rmd x_2
\end{eqnarray*}
We now set $u_1=2^{j_1}x_1-k_1$, $u_2=2^{j_2}x_2-k_2$ and deduce that
\[
|c_{j_1,j_2,k_1,k_2}|\leq \left(2^{-\frac{j_1 s}{\alpha_1}}\int_{\mathbb{R}^2}|u_1|^{s/\alpha_1} |\psi(u_1)\psi(u_2)|\rmd u_1\rmd u_2+\left|\frac{k_1}{2^{j_1}}-a\right|^{s/\alpha_1}\int |\psi(u_1)\psi(u_2)|\rmd u_1\rmd u_2\right)\;.
\]
Hence for some $C$ depending only on $\psi$, $s$ and $\alpha$ one has
\[
|c_{j_1,j_2,k_1,k_2}|\leq C(2^{-\frac{j_1 s}{\alpha_1}}+\left|\frac{k_1}{2^{j_1}}-a\right|^{s/\alpha_1})
\]
A similar approach yields that
\[
|c_{j_1,j_2,k_1,k_2}|\leq C(2^{-\frac{j_2 s}{\alpha_2}}+\left|\frac{k_2}{2^{j_2}}-b\right|^{s/\alpha_2})
\]
This shows that (\ref{e:WC2micro}) can be read as a necessary condition for pointwise regularity of function $f$.

Let us now prove the converse result. Assuming that~(\ref{e:WC2micro}) holds, the aim first consists in defining a polynomial approximation of $f$ at $x_0$.
To that end, a Taylor expansion is used to investigate the differentiability of $f$ at $x_0$.
Let us define $f_j$ as:
\[
f_j=\sum_{(j_1,j_2)\in \Gamma_j(\alpha)}\sum_{(k_1,k_2)\in\mathbb{Z}^2}c_{j_1,j_2,k_1,k_2}\psi_{j_1,j_2,k_1,k_2}\;.
\]
where the notations are the same as in the proof of Proposition~\ref{pro:WCGlobal}. One has
\begin{eqnarray*}
|f_j(x)|&\leq& \sum_{(j_1,j_2)\in \Gamma_j}\sum_{(k_1,k_2)\in\mathbb{Z}^2}\frac{\min(2^{-j_1s/\alpha_1}+|\frac{k_1}{2^{j_1}}-a|^{s/\alpha_1},2^{-j_2 s/\alpha_2}+|\frac{k_2}{2^{j_2}}-a|^{s/\alpha_2})}
{(1+|2^{j_1}x_1-k_1|)^N(1+|2^{j_2}x_2-k_2|)^N}\\
&\leq& \sum_{j_1\leq j}\sum_{k_1,k_2}\frac{2^{-j s}+|\frac{k_2}{2^{j}}-x_2|^{s/\alpha_2}+|x_2-b|^{s/\alpha_2}}{(1+|2^{j_1}x_1-k_1|)^N(1+|2^{j_2}x_2-k_2|)^N}+\sum_{j_2\leq j}\frac{2^{-js}+|\frac{k_1}{2^{j_1}}-x_1|^{s/\alpha_1}+|x_1-a|^{s/\alpha_2}}{(1+|2^{j_1}x_1-k_1|)^N(1+|2^{j_2}x_2-k_2|)^N}
\end{eqnarray*}
Then
\begin{equation}\label{e:majofj1}
|f_j(x)|\leq C(j2^{-js}+j|x_1-a|^{s/\alpha_1}+j|x_2-b|^{s/\alpha_2})\;.
\end{equation}
In the same way, if $\beta=(\beta_1,\beta_2)$, one has
\[
|\partial^\beta f_j|\leq \sum_{(j_1,j_2)\in \Gamma_j}2^{j_1\beta_1+j_2\beta_2}\sum_{(k_1,k_2)\in\mathbb{Z}^2}\frac{\min(2^{-j_1s/\alpha_1}+|\frac{k_1}{2^{j_1}}-a|^{s/\alpha_1},2^{-j_2 s/\alpha_2}+|\frac{k_2}{2^{j_2}}-a|^{s/\alpha_2})}
{(1+|2^{j_1}x_1-k_1|)^N(1+|2^{j_2}x_2-k_2|)^N}\;.
\]
Then
\begin{equation}\label{e:majofj2}
|\partial^\beta f_j(x)|\leq C2^{j(\beta_1\alpha_1+\beta_2\alpha_2)}(2^{-js}+|x_1-a|^{s/\alpha_1}+|x_2-b|^{s/\alpha_2})\;.
\end{equation}
So, the function $f$ is $\beta$--differentiable at $x_0$ provided that $\beta_1\alpha_1+\beta_2\alpha_2\leq s$.
The Taylor polynomial of $f$ at $x_0$ is defined by
\[
P_{j,x_0}(x)=\sum_{\beta_1\alpha_1+\beta_2\alpha_2\leq s}\frac{(x-x_0)^\beta}{\beta!}\partial^\beta f_j(x_0)
\]
and
\[
P_{x_0}(x)=\sum_j P_{j,x_0}(x)\;.
\]

We shall now bound $|f(x)-P_{x_0}(x)|$ in the neighborhood of $x_0$. Recall that $f$ is assumed to be uniformly H\"{o}lder, namely there exists some $\varepsilon_0^*>0$ such that $f\in\mathcal{C}^{\varepsilon_0^*}(\mathbb{R}^2)$. The inclusions between H\"{o}lder spaces with different anisotropies (see \cite{triebel:2006}) leads to the existence of $\varepsilon_0$ such that $f\in\mathcal{C}^{\varepsilon_0,\alpha}(\mathbb{R}^2)$. Set $J_1=[\alpha J/\varepsilon_0]$. Observe that
\[
|f(x)-P_{x_0}(x)|\leq \sum_{j\leq J}|f_j(x)-P_{j,x_0}(x)|+\sum_{j= J+1}^{J_1}|f_j(x)|+\sum_{j>J_1}|f_j(x)|+\sum_{j>J}|P_{j,x_0}(x)|\;.
\]
Let us now bound each term of the right hand side of this inequality.

We first deal with the term corresponding to $j\leq J$. In this case we shall use an anisotropic version of Taylor inequality which can be found in~\cite{calderon:torchinsky:1977}, \cite{folland:stein:1982} and recalled in~\cite{benbraiek:benslimane:2011b}. It gives the existence of some $C>0$ such that
\[
|f_j(x)-P_{j,x_0}(x)|\leq C\sum_{\beta_1+\beta_2\leq k+1,\,\alpha_1\beta_1+\alpha_2\beta_2>s}|x-x_0|^{\alpha_1\beta_1+\alpha_2\beta_2}_\alpha\sup_{z=(z_1,z_2)\in\mathbb{R}^{2}} |\partial^\beta f_j|\;.
\]
with $k=[\max(s/\alpha_1,s/\alpha_2)]$.
The bound~(\ref{e:majofj2}) implies that there exists some $C>0$ such that
\[
|f_j(x)-P_{j,x_0}(x)|\leq C \sum_{\beta_1+\beta_2\leq k+1,\,\alpha_1\beta_1+\alpha_2\beta_2>s}|x-x_0|^{\alpha_1\beta_1+\alpha_2\beta_2}_\alpha 2^{j(\beta_1\alpha_1+\beta_2\alpha_2)}(2^{-js}+|x_1-a|^{s/\alpha_1}+|x_2-b|^{s/\alpha_2})
\]
Hence,
\[
\sum_{j\leq J}|f_j(x)-P_{j,x_0}(x)|\leq C \sum_{\beta_1+\beta_2\leq k+1,\,\alpha_1\beta_1+\alpha_2\beta_2>s}|x-x_0|^{\alpha_1\beta_1+\alpha_2\beta_2}_\alpha(2^{J(\beta_1\alpha_1+\beta_2\alpha_2-s)}+
2^{J(\beta_1\alpha_1+\beta_2\alpha_2)}|x-x_0|^s_\alpha)\;.
\]
Since $|x-x_0|_\alpha\leq 2^{-J}$ it comes
\begin{equation}\label{e:ineq1}
\sum_{j\leq J}|f_j(x)-P_{j,x_0}(x)|\leq C |x-x_0|^s_\alpha
\end{equation}

Let us now bound the sum $\sum_{j= J+1}^{J_1}|f_j(x)|$. By~(\ref{e:majofj1}) and the definition of $J_1$ which depends on $J$, one has
\begin{equation}\label{e:ineq2}
\sum_{j= J+1}^{J_1}|f_j(x)|\leq \sum_{j= J}^{J_1}(j2^{-js}+j|x-x_0|^s_\alpha)\leq J2^{-Js}+J^2|x-x_0|^s_\alpha\;.
\end{equation}

To bound the sum $\sum_{j>J_1}|f_j(x)|$ the uniform regularity of $f$ is used, leading to
\begin{equation}\label{e:ineq3}
\sum_{j>J_1}|f_j(x)|\leq C2^{-J_1\varepsilon_0}\leq C2^{-Js}
\end{equation}
the last equality following from the definition of $J_1$.

Finally, by~(\ref{e:majofj2}),  the sum $\sum_{j>J}|P_{j,x_0}(x)|$ can be bounded. Indeed,  for some $C>0$, one has
\[
\sum_{j>J}|P_{j,x_0}(x)|\leq \sum_{\beta_1\alpha_1+\beta_2\alpha_2<s}\frac{|(x-x_0)^\beta|}{\beta!}\sum_{j>J}|\partial^\beta f_j(x_0)|\leq C \sum_{\beta_1\alpha_1+\beta_2\alpha_2<s}\frac{|x_1-a|^{\beta_1}|x_2-b|^{\beta_2}}{\beta!}\sum_{j>J}2^{j(\beta_1\alpha_1+\beta_2\alpha_2-s)}
\]
Since $|x_1-a|\leq |x-x_0|_\alpha^{\alpha_1}\leq 2^{-J\alpha_1}$ and $|x_2-b|\leq |x-x_0|_\alpha^{\alpha_2}\leq 2^{-J\alpha_2}$ it comes
\begin{equation}\label{e:ineq4}
\sum_{j>J}|P_{j,x_0}(x)|\leq C \sum_{\beta_1\alpha_1+\beta_2\alpha_2<s}2^{-J(\beta_1\alpha_1+\beta_2\alpha_2)}\sum_{j>J}2^{j(\beta_1\alpha_1+\beta_2\alpha_2-s)}\leq C2^{-Js}\;.
\end{equation}

Finally, Inequalities~(\ref{e:ineq1}), (\ref{e:ineq2}), (\ref{e:ineq3}) and (\ref{e:ineq4})  yield that $f\in\mathcal{C}^{s,\alpha}_{|\log|^2}(x_0)$.
\end{proof}
Theorem~\ref{th:WL} is  a straightforward consequence of the two--microlocal criterion and of the following lemma:
\begin{lemma}
The two following properties are equivalent:
\begin{enumerate}[(i)]
\item Inequality~(\ref{e:WC2micro}) holds.
\item Inequality~(\ref{e:WL}) holds.
\end{enumerate}
\end{lemma}
\begin{proof}
Assume that~(\ref{e:WC2micro}) holds. If $\lambda'\subset 3\lambda_{j_1,j_2}(x_0)$, then
\[
j'_1\geq j_1,j'_2\geq j_2\;,
\]
and
\[
|\frac{k'_1}{2^{j'_1}}-a|\leq 2.2^{-j'_1}\mbox{ and }|\frac{k'_2}{2^{j'_2}}-b|\leq 2.2^{-j'_2}\;.
\]
Condition~(\ref{e:WC2micro}) implies
\[
|c_{\lambda'}|\leq \min (2^{-\frac{j_1s}{\alpha_1}},2^{-\frac{j_2s}{\alpha_2}})=2^{-\max(\frac{j_1}{\alpha_1},\frac{j_2}{\alpha_2})s}\;.
\]

Conversely, assume that~(\ref{e:WL}) holds. Let $\lambda'=\lambda(j'_1,j'_2,k'_1,k'_2)$ an hyperbolic dyadic cube. Set
\[
j_1=\sup\{\ell_1,\,2^{-j'_1}+|\frac{k'_1}{2^{j'_1}}-a|\leq 2^{-\ell_1}\}
\]
and
\[
j_2=\sup\{\ell_2,\,2^{-j'_2}+|\frac{k'_2}{2^{j'_2}}-b|\leq 2^{-\ell_2}\}
\]
We  have $\lambda'\subset3\lambda_{j_1,j_2}(x_0)$. Since~(\ref{e:WL}) holds one has
\[
|c_{\lambda'}|\leq \min(2^{-\frac{j_1 s}{\alpha_1}},2^{-\frac{j_2 s}{\alpha_2}})\leq C\min(2^{-\frac{j'_1 s}{\alpha_1}}+\left|\frac{k'_1}{2^{j'_1}}-a\right|^{s/\alpha_1},2^{-\frac{j'_2 s}{\alpha_2}}+\left|\frac{k'_2}{2^{j'_2}}-b\right|^{\frac{s}{\alpha_2}})\;,
\]
that is ~\ref{e:WC2micro} holds.
\end{proof}
\subsection{Proof of Theorem~\ref{th:multform}}\label{s:proofmultimulti}
The proof of Theorem~\ref{th:multform} is based on the two following lemmas, analogous to Propositions~7 and~8 of~\cite{jaffard:2004}:
\begin{lemma}\label{lem1}
Set $\alpha=(a,2-a)$ and define
\[
G(H,\alpha)=\{x\in \mathbb{R}^2,\,f\not\in\mathcal{C}^{H,\alpha}_{|\log|^2}(x)\}\;.
\]
Let $p>0$ and $s\in (0,\omega(p,\alpha)/p]$. Then for any $H\geq s-2/p$
\[
\mathrm{dim}_H(G(H,\alpha))\leq Hp-sp+2\;.
\]
If $H< s-2/p$, $\mathrm{dim}_H(G(H,\alpha))=\emptyset$.
\end{lemma}
\begin{lemma}\label{lem2}
Set $\alpha=(a,2-a)$ and define
\[
B(H,\alpha)=\{x\in \mathbb{R}^2,\,f\in\mathcal{C}^{H,\alpha}(x)\}\;.
\]
Let $p<0$ and $s\in (0,\omega(p,\alpha)/p]$. Then
\[
\mathrm{dim}_H(B(H,\alpha))\leq \mathrm{dim}_P(B(H,\alpha)) \leq Hp-sp+2\;.
\]
\end{lemma}

The proof of Lemma~\ref{lem1} in the case $H\geq s-2/p$ is exactly the same as this of Proposition~7 of \cite{jaffard:2004}, except that the set $G_{j,H}$ are replaced with the sets
\[
G(j,H,\alpha)=\{\lambda=\lambda(j_1,j_2,k_1,k_2),\, (j_1,j_2)\in \Gamma_j(\alpha),|d_\lambda|\geq 2^{-jHp}\}\;.
\]
Lemma~\ref{lem1} in the case $H< s-2/p$, comes from the hyperbolic wavelet characterization of anisotropic Besov spaces stated in Theorem~\ref{th:WCBesov} and the Sobolev embeddings which can be proved in the anisotropic case as in the isotropic one (see~\cite{triebel:2006}).

The proof of Lemma~\ref{lem2} is exactly the same as this of Proposition~8 of \cite{jaffard:2004}, except that the set $B_{H}$ are replaced with the sets $B(H,\alpha)$.

Lemmas~\ref{lem1} and \ref{lem2} then imply Theorem~\ref{th:multform}, since for any $\alpha=(\alpha_1,\alpha_2)$ such that $\alpha_1+\alpha_2=2$ one has
\[
E(H,\alpha)\subset \left(\cap_{H'>H}G(H',\alpha)\right)\cap\left(\cup_{H'<H}B(H',\alpha)\right)\;.
\]

{\bf Acknowledgements.} We warmly thank Florent Autin and Jean Marc Freyermuth for many stimulating discussions about applications of non parametric statistics to the analysis of anisotropic textures as well as Laurent Duval for giving us very interesting additional references about hyperbolic wavelet analysis.
\bibliographystyle{siam}
\bibliography{Hyperbolic}
\end{document}